\RequirePackage{tikz}
\documentclass[12pt]{amsart}
\usepackage{amssymb,amsmath,mathrsfs, ,setspace,enumitem, xcolor} 
\usepackage{setspace}
\usepackage{bold-extra}
\usepackage{gensymb}
\usepackage{tabularx}
\usepackage{extarrows}
\usepackage{booktabs}
\usepackage{hyperref}
\usepackage{bm}
\usetikzlibrary{fadings}
\usetikzlibrary{patterns}
\usetikzlibrary{shadows.blur}
\usetikzlibrary{shapes}

\makeatletter \oddsidemargin.9375in \evensidemargin \oddsidemargin
\def\authorsaddresses#1{\dedicatory{#1}}

\newcommand{\dis}{\displaystyle}
\newcommand{\nn}{\neg\neg}

\newtheorem{proposition}{\textbf{Proposition}}[section]
\newtheorem{definition}{\textbf{Definition.}}[section]
\newtheorem{example}{\textbf{Example.}}[section]
\newtheorem{corollary}{\textbf{Corollary.}}[section]
\newtheorem{lemma}{\textbf{Lemma.}}[section]
\newtheorem{theorem}{\textbf{Theorem.}}[section]
\newtheorem{remark}{\textbf{Remark.}}[section]

\DeclareOldFontCommand{\rm}{\normalfont\rmfamily}{\mathrm}%
\DeclareOldFontCommand{\sf}{\normalfont\sffamily}{\mathsf}%
\DeclareOldFontCommand{\tt}{\normalfont\ttfamily}{\mathtt}%
\DeclareOldFontCommand{\bf}{\normalfont\bfseries}{\mathbf}%
\DeclareOldFontCommand{\it}{\normalfont\itshape}{\mathit}%
\DeclareOldFontCommand{\sl}{\normalfont\slshape}{\@nomath\sl}%
\DeclareOldFontCommand{\sc}{\normalfont\scshape}{\@nomath\sc}%
\DeclareRobustCommand*\cal{\@fontswitch{\relax}{\mathcal}}%
\DeclareRobustCommand*\mit{\@fontswitch\relax\mathnormal}%

\usepackage[top=3cm, left=2cm, right=2cm, bottom=2.5cm]{geometry}

\begin{document}
\setcounter{page}{1}


\title[Some Epistemic Extensions of G\"odel Fuzzy Logic]{Some Epistemic Extensions of G\"odel Fuzzy Logic}

\author[D. Dastgheib, H. Farahani, A.H. Sharafi]{D. Dastgheib$^1$, H. Farahani$^1$, A.H. Sharafi$^2$}


\authorsaddresses{$^1$Department of Computer Science, Shahid Beheshti University, G.C, Tehran, Iran\\
	h$_-$farahani@sbu.ac.ir, D\_Dastgheib@sbu.ac.ir \\
\vspace{0.5cm}
$^2$Department of Mathematics, Tafresh University, Tafresh, Iran\\
sharafi@tafreshu.ac.ir
}


\begin{abstract}
	In this paper we prove soundness and completeness of some epistemic extensions of G\"odel fuzzy logic, based on Kripke models in which both propositions at each state and accessibility relations take values in [0,1]. We adopt belief as our epistemic operator, acknowledging that the axiom of Truth may not always hold. 
		We propose the axiomatic system $\textbf{K}_\textbf{F}$ serves as a fuzzy variant of  classical epistemic logic $\textbf{K}$, then by considering consistent belief and adding positive introspection and Truth axioms to the axioms of $\textbf{K}_\textbf{F}$, the  axiomatic extensions $\textbf{B}_\textbf{F}$ and $\textbf{T}_\textbf{F}$ are established. 
		To demonstrate the completeness of $\textbf{K}_\textbf{F}$, we present a novel approach that characterizes formulas semantically equivalent to $\perp$ and we introduce a grammar describing formulas with this property. Furthermore, it is revealed that validity in $\textbf{K}_\textbf{F}$ cannot be reduced to the class of all models having crisp accessibility relations, and also  $\textbf{K}_\textbf{F}$ does not enjoy the finite model property.
		These properties distinguish $\textbf{K}_\textbf{F}$ as a new modal extension of  G\"odel fuzzy logic which differs from  the standard G\"odel Modal Logics $\bm{\mathcal{G}}_\Box$ and $\bm{\mathcal{G}}_\Diamond$ proposed by Caicedo and O. Rodriguez.
\end{abstract}
\maketitle



	\section{Introduction}

			The study of \textit{fuzzy} epistemic logic is indispensable due to the inherent fuzziness of belief and knowledge. Each person's understanding and conviction about something differ, and utilizing real numbers within the range of $[0,1]$ as degrees of belief or knowledge proves to be advantageous. This approach offers fresh insights into the concepts of belief and knowledge. For instance, Gettier's counterexamples to Plato's definition of knowledge, known as ``justified true belief," \cite{Gettier1963} may not pose a significant problem. In other words, to address Gettier's challenges, it appears appropriate to adopt a fuzzy interpretation of Plato's definition.
			
		The literature contains several modal extensions of fuzzy logics. For example, \cite{CMRR13, Caicedo2004, CR10, CR15, Diequez2023} present modal extensions of G\"odel fuzzy logic. Additionally, \cite{Corsi2023, Hansoul2006, Hansoul2013} investigate modal extensions of {\L}ukasiewicz logic, employing Kripke-based semantics with classical accessibility relations. In the study conducted in \cite{product2015}, modal extensions of product fuzzy logic are examined using both relational and algebraic semantics. The relational semantics in this case are based on Kripke structures with classical accessibility relations. Moreover, \cite{Hajek2010} introduces a fuzzy variant of recursively axiomatized logics that extend $S_5$, proposed by
		H\'{a}jek.
		Furthermore, \cite{Thesis2015} studies modal logics over MTL, where their semantics are based on Kripke structures with truth values ranging over [0,1] and classical accessibility relations.

	In this paper, our basic epistemic extension of G\"odel fuzzy logic is $\textbf{K}_\textbf{F}$ in which the corresponding semantics is based on Kripke structures  with both fuzzy truth values and fuzzy accessibility relations. $\textbf{K}_\textbf{F}$ serves as a fuzzy variant of classical epistemic logic $\textbf{K}$. Similarly, in \cite{CR10}, the standard G\"odel modal logics $\bm{\mathcal{G}}_\Box$ and $\bm{\mathcal{G}}_\Diamond$ over the standard G\"odel algebra are introduced. In our study, we approach it differently by considering belief $B$ as the modal operator and by introducing a semantics based on Kleene-Dienes implication. This approach seems to enhance an agent's belief representation in real-world scenarios.
	We can see, in contrast to  $\bm{\mathcal{G}}_\Diamond$, the logics $\textbf{K}_\textbf{F}$ and $\bm{\mathcal{G}}_\Box$ both enjoy the finite model property. Additionally unlike  $\bm{\mathcal{G}}_\Box$, the validity in $\textbf{K}_\textbf{F}$ and $\bm{\mathcal{G}}_\Diamond$ requires truly fuzzy accessibility relations. 
	Relying on \cite{CR10}, G\"odel justification logic was proposed in \cite{Pischke2020}. Recently, completeness of many-valued G\"odel counterpart of the classical modal logic KD45 and K45 extensions of G\" odel fuzzy logic was provided in \cite{Rodriguez2022}.
		
	In the following, we provide a collection of axioms and rules from traditional modal logics (refer to \cite{ditmarsch2015}). These axioms and rules serve as the foundation for proposing our desired logics. \\
		1.	$B_a(\varphi\rightarrow\psi)\rightarrow(B_a\varphi\rightarrow B_a\psi)$ (\textbf{distribution of $B_a $ over implication}),\\
		2.	from $\varphi$ and $\varphi\rightarrow\psi$ infer $\psi$ (\textbf{modus ponens}),\\
		3.	from $\varphi$ infer $B_a\varphi$ (\textbf{necessitation of belief}),\\
		4.	$\neg B_a\bot$ (\textbf{consistent belief}),\\
		5.	$B_a\varphi\rightarrow B_a B_a\varphi$ (\textbf{positive introspection}),\\
		6.	$\neg B_a\varphi\rightarrow B_a\neg B_a\varphi$ (\textbf{negative introspection}),\\
		7.	$B_a\varphi\rightarrow\varphi$ (\textbf{Truth}),\\
		8. $\varphi\rightarrow B_a\neg B_a\neg\varphi$ \textbf{(Brouwer's axiom)},\\		
		where, for every agent $a$, the formula $B_a \varphi$ is interpreted as “agent $a$ believes in formula $\varphi$".
		

	In Section 2, we present some epistemic extensions of G\"odel logic. Firstly, we introduce a unified language for these epistemic G\"odel logics. Next, we provide a Kripke-based semantics, where propositions at possible states and accessibility relations take fuzzy values within the range of [0,1]. 		
				We demonstrate the validity of consistent belief, positive introspection, and Truth in the proposed fuzzy Kripke models, where the relations are serial, transitive, and reflexive, respectively. Additionally, we introduce a fuzzy adaptation of the well-known muddy children puzzle and utilize it to illustrate the invalidity of certain schemes. 
Moreover, this fuzzy variant of the Muddy children problem sheds light on considering Belief as our modal operator.
				
	However, finding a correspondence between negative introspection and Brouwer's axiom with some properties of fuzzy accessiblity relations is still an open problem.

%
%
%
%

		In comparison to the logic $\bm{\mathcal{G}}_\Box$ proposed in \cite{CR10}, where both $K_\Box$: ($\Box(\varphi\rightarrow \psi) \rightarrow (\Box \varphi \rightarrow \Box \psi)$) and $Z_\Box$: ($\neg \neg \Box \varphi \rightarrow \Box \neg \neg \varphi$) are added to G\"odel logic, in $\textbf{K}_\textbf{F}$ we only include the axiom K: $B(\varphi\rightarrow \psi) \rightarrow (B \varphi \rightarrow B \psi)$. 
		 As a result, $\textbf{K}_\textbf{F}$ is a weaker logic than $\bm{\mathcal{G}}_\Box$. 	
		Next, we extend the axiomatic system $\textbf{K}_\textbf{F}$ by incorporating the axioms of \textit{consistent belief} and \textit{positive introspection}. This new axiomatic system is referred to as the \textit{logic of fuzzy belief} and is denoted as $\textbf{B}_\textbf{F}$.
		 Furthermore, we propose the system $\textbf{T}_\textbf{F}$ by adding the \textit{Truth} axiom to $\textbf{B}_\textbf{F}$, allowing the operator $B$ to be interpreted as knowledge in $\textbf{T}_\textbf{F}$.

	In Section 3, we establish the soundness and semi-completeness of these axiomatic systems with respect to the corresponding Kripke models. By semi-completeness, we mean that if a formula $\varphi$ is valid, then $\neg\neg\varphi$ is provable. Furthermore, we demonstrate the property of strong semi-completeness.
	
	In Section 4, we see that traditional Lindenbaum's method cannot be employed to prove the completeness of the desired logics. This is because there are certain non-valid formulas $\varphi$ that hinder the construction of any maximal and consistent set containing $\neg \varphi$. These formulas $\varphi$ possess the following properties:
	
	$\bullet$ For all models $M$ and all states $s$, the value of $\varphi$ in state $s$ of model $M$ is \textit{strictly} greater than 0.
	
	$\bullet$ There exists a model $M'$ and a state $s'$ in that model where the value of $\varphi$ in state $s'$ is \textit{strictly} less than 1.

			 It is important to note that for such a formula $\varphi$, we have $\neg \varphi \equiv \bot$.
			 We propose a new approach by introducing three classes: $\mathcal{T}$, $\mathcal{E}$, and $\mathcal{O}$, which categorize formulas whose negations are equivalent to $\perp$. We prove that if a formula $\varphi$ is valid, then $\neg \neg \varphi \rightarrow \varphi$ is provable.			 
			 Then, by utilizing the semi-completeness theorem, we establish the completeness theorem for the axiomatic systems $\textbf{K}_\textbf{F}$, $\textbf{B}_\textbf{F}$, and $\textbf{T}_\textbf{F}$ in relation to appropriate classes of Kripke models. 
			 
			 In Section 5, we give an example of a formula which is valid in all crisp models but is not valid in the class of all Epistemic G\"odel logic models. Furthermore, we demonstrate that the schema $\neg \neg B\neg \neg \varphi\rightarrow \neg \neg B \varphi$ is not valid in a given infinite model; however, we establish its validity within the class of all finite models. Therefore, we illustrate that EGL does not possess the finite model property.


		In this paper, we assume that the reader is already acquainted with the fundamental principles of classical modal logic (see, for instance, \cite{handbook}).
		Furthermore, our primary references for G\"odel logic are \cite{Hajek1998} and \cite{BaazPreining2011}, where we rely on the following theorems of G\"odel logic to support our proofs:
		$$
		\begin{array}{lll}
			&[(G1)] &(\varphi\wedge\psi)\rightarrow\varphi ,\\
			&[(G2)]& (\varphi\rightarrow(\psi\rightarrow\chi))\leftrightarrow((\varphi\wedge\psi)\rightarrow\chi),\\
			&[(GT1)] & \psi\rightarrow(\varphi\rightarrow\psi),\\
			&[(GT2)]& (\varphi\rightarrow\psi)\rightarrow(\neg\psi\rightarrow\neg\varphi),\\
			&[(GT3)]& \neg(\varphi\wedge\psi)\rightarrow(\varphi\rightarrow\neg\psi),\\
			&[(GT4)] &(\neg\neg\varphi\wedge\neg\neg\psi)\leftrightarrow\neg\neg(\varphi\wedge\psi),\\
			&[(GT5)] &\neg\neg(\varphi\rightarrow\psi)\leftrightarrow(\neg\neg\varphi\rightarrow\neg\neg\psi), \\
			&[(GT6)]&\varphi\rightarrow\neg\neg\varphi,\\
			&[(GT7)]&\neg\varphi\leftrightarrow\neg\neg\neg\varphi,\\
			&[(GT8)]&\varphi\wedge\neg\varphi\leftrightarrow\bot,\\
			&[(GT9)]&\varphi \rightarrow (\psi \rightarrow(\varphi \wedge \psi)).
	\end{array}
$$

	Throughout this paper, we will make the assumption that $\mathcal{P}$ represents a set of atomic propositions, and $\mathcal{A}$ represents a set of agents, unless otherwise specified.
		
	\section{Epistemic G\"odel logic}
	
	In this section, we present some epistemic extensions of G\"odel fuzzy logic. Initially, we introduce a language specifically designed for these logics, which we refer to as the Epistemic G\"odel Logics language, or simply EGL-language. Subsequently, we establish Kripke-based semantics for this language, wherein the values of propositions and accessibility relations are fuzzy. Finally, we propose a set of axiomatic systems over this language.

	\subsection{Semantics of Epistemic G\"odel Logic}
		
	\begin{definition}
		The language of \textit{Epistemic G\"odel Logic} (EGL-language) is defined using the following Backus-Naur Form (BNF):
		$$\varphi\ ::=\ p\mid\bot\mid\varphi\wedge\varphi\mid\varphi\rightarrow\varphi\mid B_a{\varphi}$$
		where, $p\in\mathcal{P}$ and $a\in\mathcal{A}$. 
		Note that the language of EGL is an expansion of the language of G\"odel logic (GL). For each $a\in \mathcal{A}$, we add an epistemic operator $B_a$ to the language of GL. 
		Further connectives $\neg$, $\vee$ and $\leftrightarrow$ are defined similar to G\"odel logic (See \cite{Hajek1998}) as follows:
		\begin{align*}
			\neg\varphi=&\varphi\rightarrow\bot,\\
			\varphi\vee\psi=&((\varphi\to\psi)\to\psi)\wedge((\psi\to\varphi)\to\varphi),\\
			\varphi\leftrightarrow\psi=&(\varphi\to\psi)\wedge(\psi\to\varphi).
		\end{align*}
	\end{definition}
 
	\begin{definition}(\textbf{EGL-Model})\label{label1}		 
	An EGL-model is a structure $M=(S,r_{a{\mid_{a\in\mathcal{A}}}},\pi)$, where
	\vspace{-0.2cm}

	\begin{itemize}
		\item[$\bullet$]$S$ is a non-empty  countable set of states, 
		\item[$\bullet$]$r_{a{\mid_{a\in\mathcal{A}}}}:\ S \times S\rightarrow [0,1]$ is a function which assigns a value in the range of $[0,1]$ to each pair $(s,s')\in S \times S$. We call it the \textit{accessibility relation}. Specifically, if $r_{a{\mid_{a\in\mathcal{A}}}}(s, s')=0$, it indicates that $s'$ is not accessible from $s$,

		\item[$\bullet$]$\pi:\ S\times\mathcal{P}\rightarrow [0,1]$ is a \textit{valuation} function, where in each state $s \in S$ assigns a truth value to every atomic proposition $p\in\mathcal{P}$.
	\end{itemize}\vspace{-0.2cm}

	The valuation function $\pi$ can be naturally extended to cover all formulas and we denote the extended function by $V$. 
	\\\\
	Following the convention in \cite{CR10}, we refer to the structure $M=(S, r_{a{\mid_{a\in\mathcal{A}}}} ,\pi)$:\\
	{\bf serial}, if for all $a\in\mathcal{A}$ and all $s\in S$, there is a $s'\in S$ such that $r_a(s,s')= 1$,\\
	{\bf reflexive}, if for all $a\in\mathcal{A}$ and all $s\in S$, $r_a(s,s)=1$,\\
	{\bf symmetric}, if for all $a\in\mathcal{A}$ and all $s,s'\in S$, $r_a(s,s')=r_a(s',s)$,\\
	{\bf transitive}, if for all $a\in\mathcal{A}$ and all $s,s',s''\in S$, $$r_a(s,s'')\geq\min\{r_a(s,s'),r_a(s',s'')\}.$$

	\end{definition}
From this point forward, in cases where there is no ambiguity, we will use $B \varphi$ instead of $B_a \varphi$ and $r$ instead of $r_a$. Furthermore, for the sake of simplicity throughout this paper, we will use ``model" and ``formula" instead of ``EGL-model" and ``EGL-formula," respectively.

Let $M = (S, r, \pi)$ be a model. For each state $s \in S$ and each formula $\varphi$, we will use the notation $V_s(\varphi)$ instead of $V(s,\varphi)$. This notation is defined as follows:
	\begin{itemize}
		\item[$\bullet$]$V_s(p)=\pi(s,p)~ for ~each~ p\in\mathcal{P}$;
		\item[$\bullet$]$V_s(\bot)=0$;

		\item[$\bullet$]$V_s(\varphi\wedge\psi)=\min\{V_s(\varphi),V_s(\psi)\}$;

		\item[$\bullet$]$V_s(\varphi\rightarrow\psi)=\left\{
		\begin{array}{lr}
			1 & V_s(\varphi)\leq V_s(\psi)\\
			V_s(\psi) & V_s(\varphi)> V_s(\psi)
		\end{array}\right.
		$;
		\item[$\bullet$]$\dis V_s(B_a\varphi)=\inf_{s'\in S}\max\{1-r_a(s,s'),V_{s'}(\varphi)\}
		$,
	\end{itemize}
	\vspace{-0.2cm}
	and the semantics of $\neg$ and $\vee$ are defined as in G\"odel logic:
	\begin{itemize}
		\item[$\bullet$]$V_s(\neg\varphi)=\left\{
		\begin{array}{lr}
			0& V_s(\varphi)>0\\
			1& V_s(\varphi)=0
		\end{array}\right.
		$;
		\item[$\bullet$]$V_s(\varphi\vee\psi)=\max\{V_s(\varphi),V_s(\psi)\}$.
	\end{itemize}
	Note that $\neg\varphi$ and so $\neg\neg\varphi$ take crisp values.

	The idea behind the aforementioned definition of belief is derived from the semantics definition of belief in classical epistemic logic, which can be expressed as follows:	
	$$V_{s}(B\varphi)=1~~\text{if and only if}~~(\forall s')(r(s,s')=1 \Longrightarrow V_{s'}(\varphi)=1).$$ 
When we want to generalize implication in the above definition which takes values in $\{0,1\}$, to a fuzzy environment where the truth values belongs to the unit interval $[0,1]$, fuzzy implications are useful . These implications are essential for fuzzy logics and play a significant role in solving fuzzy relational equations, fuzzy mathematical morphology,  image processing, in defining fuzzy subsethood, etc. \cite{Helbin2019}. Moreover, the following properties seems reasonable for a fuzzy belief:
%
	 \begin{itemize}
		\item[1)] for each state and every agent, the less the access to possible states, the stronger the beliefs.
		\item[2)] for each state and every agent, the greater the truth values of a formula $\varphi$ in accessible possible states, the stronger the belief in  $\varphi$.
	\end{itemize}
Hence an appropriate implication in a fuzzy perspective can be a fuzzy implication, since for which when the value of $V_{s}(B\varphi)$ increases the value of $r(s,s')$ decreases or when the value of $V_{s'}(\varphi)$ increases.
%
%
%

 In \cite{CR10}, the authors utilized G\"odel implication 
  in the semantics definition of their modal operator,
that is $V_s^G(B\varphi) = \inf_{s'\in S}\{r(s,s')\Rightarrow_{G} V_{s'(\varphi)}\}$, where 
	\begin{align}\nonumber 
	&\Rightarrow_{G}(x,y)=\left\{
	\begin{array}{lr}
		1& x\leq y\\
		y& x>y
	\end{array}\right.
\end{align}
is the G\"odel implication.
 Instead of the  G\"odel implication, we employ \textit{Kleene-Dienes implication}:  $\Rightarrow_{KD}(x,y)=\max\{1-x,y\}$. Using the Kleene-Dienes fuzzy implication in the definition of the semantics for the belief operator $B$ provides some advantages and is considered more suitable for our purposes. For instance, when using the Kleene-Dienes implication operator, changes in the values of the accessibility relation often lead to changes in the truth value of $B (\varphi)$. In contrast, when the G\"odel implication is utilized, the values of the accessibility relation may not necessarily change, and as a consequence, the truth value of $B (\varphi)$ does not necessarily change.
   Furthermore, the subsequent example highlights the advantage of incorporating $\Rightarrow_{KD}$ into the semantics of $B$.
\begin{example}\label{colorblind}\textbf{(colorblindness)}
	Let's consider a scenario involving Alice, who is aware of her colorblindness. She is wearing a dark green left sock and is attempting to find a matching right socks from her wardrobe.
	We define the set of atomic propositions as $\mathcal{P} = \{p_i \,\mid\, i\in \{1,\cdots, 5\}\}$, where $p_i$ represents the proposition that the $i$-th right sock is dark green. For this example, let's assume that the 2nd right sock is actually light green. 
For simplicity, we use a two-state model $M$, where $\mathcal{A}=\{a\}$ represents Alice. The set of states, $S$, consists of $s_1$ and $s_2$. The accessibility relation $r_{Alice}$ and valuation function $\pi$ are depicted in Figure \ref{fig_example-colorblindness}.
In this figure, $\pi(s_2,p_i)=1$ indicates that the $i$-th right sock is completely dark green.
To compare the belief operators $I_{KD}$ and $I_G$, we find that $V_{s_2}(B_a p_2)=0.7$, while $V_{s_2}^G(B_a p_2) = 1$. The former value seems more appropriate since it suggests that Alice does not fully believe that her sock is dark green. 
	\end{example}

\begin{figure}
	\begin{center}
		\resizebox*{4cm}{!}{
			
			\begin{tikzpicture}[->,shorten >=1pt,auto,
				thick, main node/.style={circle,fill=yellow!10,draw, font=\sffamily\Large\bfseries,minimum size=10mm}]
				\node [main node] (v1) at (-3,0) {$s_1$};
				\node [main node] (v2) at (0,0) {$s_2$};
				\draw  (v1) edge[loop above, looseness=10] (v1);
				\draw  (v2) edge[loop above, looseness=10] (v2);
				\draw  (v1) edge [bend right] (v2);
				\draw  (v2) edge [bend right] (v1);
				\node at (-1.5,1) {0.5};
				\node at (-1.5,-1) {0.5};
				\node at (0.5,1) {0.5};
				\node at (-3.5,1) {0.5};
				\node at (0,-1) {$\pi(s_2,p_2)=1$};
				\node at (-3.5,-1) {$\pi_(s_1, p_2)=0.7$};
			\end{tikzpicture}
			
		}
	\end{center}
	
	\caption{Colorblindness model}\label{fig_example-colorblindness}
\end{figure}

%

As in \cite{ditmarsch2015} we have the following definition.

	\begin{definition}
		Let  $\varphi$ be a formula and  $M=(S, r_{} ,\pi)$ be a model.
    	\vspace{-0.2cm}
		\begin{itemize}
			\item[(1)] 	For a given state $s \in S$, we say that a formula $\varphi$ is true in the pointed model $(M, s)$, denoted as $(M, s) \vDash \varphi$, if $V_s(\varphi) = 1$. Here, the pair $(M, s)$ is referred to as a pointed model, representing a specific state $s$ within the model $M$.

			\item[(2)] We say that a formula $\varphi$ is true in model $M$, denoted as $M \vDash \varphi$, if for every state $s' \in S$, $(M, s') \vDash \varphi$. This notation signifies that $\varphi$ holds true in $M$ when it is true in each state $s'$ of the model $(M, s')$.

			\item[(3)] If $\mathcal{M}$ is a class of models, we say that a formula $\varphi$ is $\mathcal{M}$-valid, denoted as $\mathcal{M}\vDash\varphi$, if for every model $M' \in \mathcal{M}$, we have $M'\vDash\varphi$. In other words, $\varphi$ is $\mathcal{M}$-valid if it holds true in every model $M'$ belonging to the class $\mathcal{M}$.

			\item[(4)] A formula $\varphi$ is considered EGL-valid or simply valid, denoted as $\vDash\varphi$, if it holds true in every model $M$. In other words, for any model $M$, we have $M\vDash\varphi$. This notation signifies that $\varphi$ is universally true across all possible models.
		\end{itemize}

	\end{definition}

	In the following, we give a fuzzy version of traditional muddy children problem and model it using EGL.

	\begin{example} \textbf{(A fuzzy muddy children)}\label{MuddyChildren} 

	In the classic form of the muddy children puzzle, it is assumed that a group of children are playing outside and some of them have muddy foreheads. Each child can observe whether the foreheads of the other children are muddy or not, but they do not have knowledge about their own forehead.
		
	To model this puzzle in a fuzzy format, we suggest introducing a ``degree of muddiness" that represents the extent of muddiness on a scale of real numbers over $[0,1]$. By utilizing Definition \ref{label1} and leveraging the principles of fuzzy logic, we can create a fuzzy version of the muddy children problem.

		Let's consider a scenario with $k$ children, represented by the set of agents $\mathcal{A} = \{a_1,\cdots,a_k\}$. We define the set of atomic propositions as $\mathcal{P}= \{m_i \mid 1\leq i \leq k\}$, where each atomic proposition $m_i$ corresponds to the agent $a_i$ and represents the statement ``The forehead of agent $a_i$ is muddy."
		In this model, each state is defined as a $k$-tuple $(t_1,\cdots,t_k)$, where $t_i \in [0, 1]$ represents the degree of muddiness of the $i$-th child's forehead. More formally, we can denote the set of all possible states as $S=\{(t_j)_{1\leq j\leq k}\mid t_j \in [0,1]\}$.

		In this model, we define the valuation function $\pi:S\times \mathcal{P}\rightarrow [0,1]$ as follows:
		$$\pi (\,(t_j)_{1\leq j \leq k}\, , m_i ) = t_i \qquad(1\leq i\leq k),$$
	this means that the truth value of the proposition $m_i$ in the state $(t_j)_{1\leq j \leq k}$ is given by the corresponding degree of muddiness  $t_i$. For example, if we have three agents and consider the state  $(0, 0.5, 0.75)$, we can evaluate the truth value of the proposition $m_3$ as $\pi((0, 0.5, 0.75), m_3)=0.75$.
		
		Let $s_1 = (t_j^1)_{1\leq j\leq k}$ and $s_2 = (t_j^2)_{1\leq j\leq k}$ be two states in $S$. In addition, suppose that for each agent $a\in \mathcal{A}$, $\mathcal{B}_a\in [0,1]$ denotes the degree of visual impairment of an agent $a$. Now for each agent $a_i$, $(1\leq i\leq k)$, we define the accessibility relation $r_{a_i}$ as follows:
		$$r_{a_i}(s_1,s_2)=\min\Big\{r_{a_i a_j}^{s_1 s_2}\mid~ 1\leq j\leq k, j\neq i\Big\}\quad$$
		in which, for each $~1\leq j\leq k,~j\neq i$ we have
		$$r_{a_i a_j}^{s_1 s_2}= \left\{
		\begin{array}{lr}
			{\mathcal{B}}_{a_i}(1-\alpha \vert t_j^1-t_j^2 \vert )&  t_j^1\neq t_j^2\\
			{\mathcal{B}}_{a_i}& t_j^1= t_j^2
		\end{array}\right.$$
		where  $\alpha\in(0,1)$ is a tuning parameter.
		 It is important to note that for an agent $a_i$,  we can see 
		from the definition of $r_{a_i a_j}^{s_1 s_2}$ that the greater the difference between the amount of mud on the forehead of agent $a_j$ in the states $s_1$ and $s_2$ (i.e. greater $\vert t_j^1-t_j^2\vert$), the smaller the value of $r_{a_i a_j}^{s_1 s_2}$. This implies that agent $a_i$ can  better distinguish the $j^{th}$ element of states $s_1$ and $s_2$. 		
		Additionally, it's worth mentioning that by  assigning different values to the tuning (rational) parameter $\alpha$ we can control the impact of $\vert t_j^1-t_j^2\vert$. For example  decreasing the value of  $\alpha$,  reduces the influence of muddiness and increases the effect of blindness

\end{example}


In the following proposition, we use the aforementioned example to illustrate that consistent belief, \textit{positive introspection}, and \textit{Truth} lack validity.

\begin{proposition} \label{invalid}The following schemas are not valid. 
\begin{itemize}
    \item[(1)]$\neg B\bot$ (\textit{Consistent Belief})
	\item[(2)]$B\varphi\rightarrow B B\varphi$ (\textit{Positive Introspection})
	\item[(3)]$B \varphi \rightarrow \varphi$ (Truth)
\end{itemize}

\end{proposition}

\begin{proof} We construct a model $M'$ based on Example \ref{MuddyChildren} to demonstrate that the desired schemas are not valid. Let 
$\mathcal{A}=\{a,b\}$, ${\mathcal{B}}_a=1$, ${\mathcal{B}}_b=0.9$ and $\alpha=0.2$. In this model, the set of states is given by $S=\Big\{(i,j)\mid i,j\in\{0,0.5,1\}\Big\}$. For arbitrary states $s_1=(x,y)$ and $s_2=(x',y')$ in $S$, the accessibility relations $r_a$ and $r_b$ are computed as follows:

$$	r_{a}((x,y),(x',y'))=\left\{
\begin{array}{ll}
	1&\text{ if } \vert y-y' \vert=0\\
	0.9&\text{ if } \vert y-y'\vert =0.5\\
	0.8&\text{ if } \vert y-y'\vert =1
\end{array}\right.
\;
r_b((x,y),(x',y'))=\left\{
\begin{array}{ll}
	0.9& \text{ if }\vert x-x'\vert =0\\
	0.81&\text{ if } \vert x-x'\vert =0.5\\
	0.72&\text{ if } \vert x-x'\vert =1
\end{array}\right.\vspace{-0.2cm}$$
\vskip 0.2cm
Now we provide instances of schemas  (1) to (3) that are refuted to be valid in $M'$. For each schema we find a state $s\in S$ and a formula $\varphi$ such that $(M',s)\nvDash\varphi$.

	\begin{itemize}
\item[(1)]
	We show that $(M', (0,0))\nvDash \neg B_b \perp$. It is easy to check that

			$$ \begin{array}{ll}
			V_{(0,0)} (B_b \perp) & = \inf_{s=(x,y)\in S} \max\{1-r_b((0,0), (x,y)), V_{(x,y)}(\perp)\} \\
			& =  \inf \{\max \{0.1, 0\}, \max\{0.19, 0\}, \max\{0.28, 0\}\} \\
			& = 0.1,
		\end{array} $$
	
 and so $V_{(0,0)}(\neg B_b \perp) = 0$
\item[(2)]We have  $(M',(1,1))\nvDash B_b m_a\rightarrow B_bB_b m_a$. 
The computation of $V_{(1,1)} (B_b B_b m_a)$ requires the value of $V_{(x,y)} (B_b m_a)$ for any pair $(x,y)$.
We therefore first compute those, with the computation of $V_{(0,0)} (B_b m_a)$ as an example. The required
value of $V_{(1,1)} (B_b m_a)$ is then another special case.

$$\begin{array}{ll}
	V_{(0,0)}(B_b m_a)  & = \inf_{s=(x,y)\in S}\max\{1-r_b((0,0),(x,y)), V_{(x,y)}(m_a)\}\\ 
	& = \inf_{s=(x,y)\in S} \{\\
	& \qquad \max \{1-r_b((0,0), (0,y)), V_{(0,y)}(m_a)\}\\ 
	& \qquad \max \{1-r_b((0,0), (0.5,y)), V_{(0.5,y)}(m_a)\}\\
	& \qquad \max \{1-r_b((0,0), (1,y)), V_{(1,y)}(m_a)\}\}\\
	& = \inf\{ \max\{0.1, 0\}, \max\{0.19, 0.5\}, \max\{0.28, 1\} \} \\
	& = 0.1.
\end{array}$$
By similar computation for other states we have:

\begin{itemize}
	\item[] $V_{(0,0)}(B_b m_a)=V_{(0,0.5)}(B_b m_a)=V_{(0,1)}(B_b m_a)=0.1,$
	\item[]$V_{(0.5,0)}(B_b m_a)=V_{(0.5,0.5)}(B_b m_a)=V_{(0.5,1)}(B_b m_a)=0.19,$
	\item[]$V_{(1,0)}(B_b m_a)=V_{(1,0.5)}(B_b m_a)=V_{(1,1)}(B_b m_a)=0.28.$
\end{itemize}

Therefore,
$$\begin{array}{ll}
	V_{(1,1)}(B_bB_b m_a) & = \inf_{s=(x,y)\in S} \max\{1-r_b((1,1),(x,y)), B_b m_a\} \\
	& = \inf_{s=(x,y)\in S} \{\\
	&\qquad \max\{0.1, V_{(1,y')} (m_a) \},\\
	&\qquad \max\{0.19, V_{(0.5,y')} (m_a) \},\\
	&\qquad \max\{0.28, V_{(0,y')}(m_a) \}\}\\
	& = \inf \{ \max\{0.1, 0.28\}, \max\{0.19, 0.19\}, \max\{0.28, 0.1\}\} \\
	& = 0.19.
\end{array}$$
 Now, since 
$$V_{(1,1)}(B_b m_a)>V_{(1,1)}(B_bB_b m_a),$$ by definition, it can be obtained that 
$$V_{(1,1)}(B_b m_a\rightarrow B_bB_b m_a) = V_{(1,1)}(B_bB_b m_a)  = 0.19.$$
	\end{itemize} 

Similarly, we provide counterexamples that refute the validity of scheme (3):  
\begin{itemize}
	\item[(3)] We have $V_{(0,0)} (m_a) = \pi((0,0), m_a) = 0$. Since $V_{(0,0)} (B_b m_a) \leq V_{(0,0)} (m_a)$. Thus	
	$V_{(0,0)} (B_b m_a \rightarrow m_a)=0$.
\end{itemize}\vspace{-0.6cm}
\end{proof}

\begin{proposition}\label{sound} Let $\varphi$ be a formula. The following formulas are valid. 
\vspace{0.2cm}
\\
$~~(i$) $(B\varphi\wedge B(\varphi\rightarrow\psi))\rightarrow B\psi$;\\
$~~(ii$) $B(\varphi\rightarrow\psi)\rightarrow(B\varphi\rightarrow B\psi)$\\
$~~(iii$) $B(\varphi\wedge\psi)\leftrightarrow (B\varphi\wedge B\psi)$;\\
$~~(iv$) $(B\varphi\wedge B\neg\varphi)\leftrightarrow B\bot$.
\end{proposition}

\begin{proof} 
(i): Suppose that  $M=(S, r,\pi)$ is a model and $s\in S$. We show that $(B\varphi\wedge B(\varphi\rightarrow\psi))\rightarrow B\psi$ is true in the pointed model $(M,s)$.  Let  $\Gamma=\{s'\in S\mid V_{s'}(\varphi)> V_{s'}(\psi)\}$, then
if $\Gamma=\phi$, we have $V_s(B\varphi\wedge B(\varphi\rightarrow\psi))=V_s(B\varphi)$ and also
$(\forall t\in S) (V_t(\varphi)\leq V_t(\psi))$.
So, we can conclude that $(\forall s\in S) (V_s(B\varphi)\leq V_s(B\psi))$. Hence, 
$V_s(B\varphi\wedge B(\varphi\rightarrow\psi))\leq V_s(B\psi)$ and so 
$V_s((B\varphi\wedge B(\varphi\rightarrow\psi))\rightarrow B\psi)=1$. Otherwise, if $\Gamma\neq\phi$, 
 then
\begin{align*}
	V_s&(B\varphi\wedge B(\varphi\rightarrow\psi))\\
	&=\min\bigg\{\min\Big\{\inf_{s'\in \Gamma}\max\{1-r(s,s'),V_{s'}(\varphi)\},\inf_{s'\in \Gamma^c}\max\{1-r(s,s'),V_{s'}(\varphi)\}\Big\},\\
	&\qquad\quad\inf_{s'\in \Gamma}\max\{1-r(s,s'),V_{s'}(\psi)\}\bigg\}\\
	&=\min\bigg\{\inf_{s'\in \Gamma}\max\{1-r(s,s'),V_{s'}(\varphi)\},\inf_{s'\in \Gamma^c}\max\{1-r(s,s'),V_{s'}(\varphi)\},\\ 
	&\qquad\quad\inf_{s'\in \Gamma}\max\{1-r(s,s'),V_{s'}(\psi)\}\bigg\}.
\end{align*}
Now, if $s'\in\Gamma$, then $V_{s'}(\psi)<V_{s'}(\varphi)$ and so 
$$\inf_{s'\in \Gamma}\max\{1-r(s,s'),V_{s'}(\psi)\}\leq\inf_{s'\in \Gamma}\max\{1-r(s,s'),V_{s'}(\varphi)\} \vspace{-0.2cm}$$
Therefore,
\begin{align*}
V_s&(B\varphi\wedge B(\varphi\rightarrow\psi))=\\
&\min\Big\{\inf_{s'\in \Gamma}\max\{1-r(s,s'),V_{s'}(\psi)\},\inf_{s'\in \Gamma^c}\max\{1-r(s,s'),V_{s'}(\varphi)\}\Big\}.\nonumber
\end{align*}
Also, if $s'\in \Gamma^c$, then $V_{s'}(\varphi)\leq V_{s'}(\psi)$  and so 
$$\inf_{s'\in \Gamma^c}\max\{1-r(s,s'),V_{s'}(\varphi)\}\leq\inf_{s'\in \Gamma^c}\max\{1-r(s,s'),V_{s'}(\psi)\}. \vspace{-0.2cm}$$
Consequently, we can conclude that $V_s(B\varphi\wedge B(\varphi\rightarrow\psi))\leq V_s(B\psi)$,
which completes the proof.\\

(ii): This is equivalent to (i) by (G2).

(iii): Suppose that $M=(S, r_{},\pi)$  is an arbitrary model and $s\in S$. In order to show that 
$$V_s(B (\varphi \wedge \psi) \leftrightarrow (B \varphi \wedge B \psi)) =1,$$
we prove that $V_s(B (\varphi \wedge \psi)) = V_s(B \varphi \wedge B \psi)$ in two cases: \\
case1: The inequality $V_s(B (\varphi \wedge \psi)) \leq V_s(B \varphi \wedge B \psi)$ follows from $V_s(B(\varphi\wedge\psi))\leq V_s(B\varphi)$, $V_s(B(\varphi\wedge\psi))\leq V_s(B\psi)$ and so $V_s(B(\varphi\wedge\psi)\leq \min\{V_s(B\varphi),V_s(B\psi)\}$.\\
case2: 
 By definition, we have
\begin{align*}
V_s(B(\varphi \wedge \psi))& = \inf_{s'\in S} \max\{1-r(s,s'),\min\{V_{s'}(\varphi), V_{s'}(\psi)\}\},\ \text{and}\\
V_s(B\varphi\wedge B\psi)& =\min\{\inf_{s'\in S}\max\{1-r(s,s'),V_{s'}(\varphi)\},\inf_{s'\in S}\max\{1-r(s,s'),V_{s'}(\psi)\}\}.
\end{align*}
Now, if we define $\Gamma=\{s'\in S\mid V_{s'}(\varphi)> V_{s'}(\psi)$, the definition of $V_s(B(\varphi\wedge\psi))$ can be rewriten as:
$$V_s(B(\varphi\wedge\psi))=\min\{\inf_{s'\in \Gamma}\max\{1-r(s,s'),V_{s'}(\psi)\},\inf_{s'\in \Gamma^{c}}\max\{1-r(s,s'),V_{s'}(\varphi)\}\}.$$
Then using the facts:
\begin{align*}
\inf_{s'\in S}\max\{1-r(s,s'),V_{s'}(\psi)\}&\leq\inf_{s'\in\Gamma}\max\{1-r(s,s'),V_{s'}(\psi)\},\ \text{and}\\
\inf_{s'\in S}\max\{1-r(s,s'),V_{s'}(\varphi)\}&\leq\inf_{s'\in\Gamma^{c}}\max\{1-r(s,s'),V_{s'}(\varphi)\},
\end{align*}
so  we can conclude
$$V_s(B\varphi\wedge B\psi)\leq V_s(B(\varphi\wedge\psi)).$$

(iv): Using (iii), we know that $V_s(B\varphi\wedge B\neg\varphi)=V_s(B(\varphi\wedge\neg\varphi))$ then by (ii) we have
$$V_s(B((\varphi\wedge\neg\varphi)\leftrightarrow\bot)\rightarrow(B(\varphi\wedge\neg\varphi)\leftrightarrow B\bot))=1.$$
Now since, by (GT8), $V_s(B((\varphi\wedge\neg\varphi)\leftrightarrow\bot))=1$
then we can conclude $V_s(B(\varphi\wedge\neg\varphi)\leftrightarrow B\bot)=1$
\end{proof}

		\vspace{-0.5cm}
		\subsection{The proof systems}
		
		First, we introduce a fuzzy version of the classical axiomatic system $\textbf{K}$ and denote  it as $\textbf{K}_\textbf{F}$. Next, we propose two extensions, namely $\textbf{B}_\textbf{F}$ and $\textbf{T}_\textbf{F}$, based on $\textbf{K}_\textbf{F}$. In the context of this discussion, let $\varphi$ and $\psi$ represent formulas. The following axiom schemas and inference rules are considered:
	\begin{eqnarray*}
		&\textbf{(A)}& \text{all instantiations of the tautologies of propositional G\"odel logic}\\
		&\textbf{(K)} & B_a(\varphi\rightarrow\psi)\rightarrow (B_a\varphi\rightarrow B_a\psi)\\
		&\textbf{(D)} & \neg B_a\bot\\
		&\textbf{(4)} & B_a\varphi\rightarrow B_a B_a\varphi\\
		&\textbf{(T)} & B_a\varphi\rightarrow\varphi\\
		&\textbf{(R1)} & \dis\frac{\varphi\qquad\varphi\rightarrow\psi}{\psi}\ (MP)\\
		&\textbf{(R2)} & \dis\frac{\varphi}{B_a\varphi}\ \ (\text{Necessitation of}\ B_a)
	\end{eqnarray*}

		Let $\textbf{K}_\textbf{F}$ denote an axiomatic system consisting of the axiom schemas \textbf{(A)} and \textbf{(K)}, along with the inference rules $\textbf{(R1)}$ and $\textbf{(R2)}$.
		We define $\textbf{B}_\textbf{F}$ as an extension of $\textbf{K}_\textbf{F}$ incorporating axioms \textbf{ (D)} and \textbf{(4)}, and we further define $\textbf{T}_\textbf{F}$ as an extension of $\textbf{B}_\textbf{F}$ including the axiom schema \textbf{(T)}.
		
		Similar to classical epistemic logic, we refer to axiom \textbf{ (D)} as \textit{consistent belief}, axiom \textbf{(4)} as \textit{positive introspection} and \textbf{(T)} as the \textit{Truth axiom}. It is worth noting that axioms \textbf{ (D)}, \textbf{(4)} and \textbf{(T)} are not valid. In Lemma \ref{A4}, we demonstrate that these axioms can be applied in serial, transitive and reflexive models, respectively. In other words, in a reflexive model, the belief $B$ can be interpreted as knowledge $K$.

	\begin{remark}
		From Proposition \ref{sound} part (iv), we can see that the axiom \textbf{(D)} is equivalent to $\neg(B\varphi\wedge B\neg\varphi)$.
	\end{remark}
	\begin{lemma} \label{A4}
		Let $M=(S, r,\pi)$ be a model.
		\begin{itemize}[nolistsep]

			\item[(i)] If $M$ is a serial model, then \textbf{(D)} is true in $M$.
			\item[(ii)] If $M$ is a transitive model, then \textbf{(4)} is true in $M$.
			\item[(iii)] If $M$ is a reflexive model, then \textbf{(T)} is true in $M$.
		\end{itemize}
	\end{lemma}

	\begin{proof}

	(i). Let $s\in S$. Then by assumption there is a $s'\in S$ such that $r(s,s')=1.$ Thus $\max\{1-r(s,s'), V_{s'}(\bot)\}=0$, and hence $\inf_{s'\in S}\max\{1-r(s,s'),V_{s'}(\bot)\}=0$, which means that $V_s(\neg B\bot)=1$ .

	(ii)
	Recall that if $M$ is transitive then for all $a\in \mathcal{A}$ and all $s,s',s'' \in S$ we have $r(s,s'')\geq \min\{r(s,s'), r(s',s'')\}$. It follows that
	\begin{equation}
		1-r(s,s'') \leq \max \{1-r(s,s'), 1-r(s',s'')\}.
	\end{equation}	 
	By definition we have $V_s(B\varphi) = \inf_{s'\in S}\max\{1-r(s,s'), V_{s'}(\varphi)\}$. Thus
	\begin{eqnarray}
		V_s(B B \varphi) && = \inf_{s'\in S}\max\{1-r(s,s'), V_{s'}(B \varphi)\}\nonumber\\
		&&=  \inf_{s''\in S}\max\{1-r(s,s'), \inf_{s''\in S}\max\{1-r(s',s''), V_{s''}(\varphi)\}\}		\label{eq_000}
	\end{eqnarray}
	Therefore from \ref{eq_000}, it is clear that $V_s(B B \varphi)$ is equal to one of the followings:\\
		\textbf{case 1.} Suppose that $V_s(B B \varphi) = V_{s_1}(\varphi)$, where $s_1$ is one of the states $s''$ in (\ref{eq_000}). Then $V_{s_1}(\varphi) \geq 1-r(s_0,s_1)$, where $s_0$ is one of the states $s'$ appeared in (\ref{eq_000}) (see figure \ref{fig_trans_prop.}). Also, $V_{s_1}(\varphi)\geq  1-r(s, s_0)$. By transitivity $1-r(s,s_1) \leq \max\{1-r(s,s_0), 1-r(s_0, s_1)\}$. Thus $V_{s_1}(\varphi)\geq 1-r(s,s_1)$, and so $V_{s_1}(\varphi)$ is one of the terms of  minimization in the formulation of  $V_s(B \varphi)$. Therefore $V_s(B \varphi) \leq V_s(B B \varphi)$.\\
		\textbf{case 2.}  Assume $V_s(B B \varphi) = 1-r(s_0, s_1)$, where $s_0$ and $s_1$ are instances of the states  $s'$ and $s''$ appeared in (\ref{eq_000}), respectively.  We have
		\begin{equation} \label{eq_001}
			1-r(s_0,s_1) \geq V_{s_1}(\varphi)
		\end{equation}
		and $1-r(s_0,s_1)\geq 1-r(s, s_0)$. Also, by transitivity we have $1-r(s,s_1) \leq \max\{1-r(s,s_0) , 1-r(s_0, s_1)\}$. Thus
		\begin{equation}\label{eq_002}
			1-r(s,s_1) \leq 1-r(s_0,s_1)
		\end{equation}
		Since $\max\{1-r(s,s_1), V_{s_1}(\varphi)\}$ is one of the terms that appears in the definition of $V_s(B \varphi)$ we have $V_s(B \varphi) \leq \max\{1-r(s,s_1), V_{s_1}(\varphi)\}$.
		Thus from (\ref{eq_001}) and (\ref{eq_002}) and the fact that $V_s(B B \varphi) = 1-r(s_0, s_1)$, we have $V_s(B \varphi) \leq V_s(B B \varphi)$. \\
		\textbf{case 3.} Let $V_s(B B \varphi) = 1-r(s,s_0)$, where $s_0$ is one of the states  $s'$ that appeared in (\ref{eq_000}). So, there is a state $s_1$ which 
		\begin{equation}\label{eq_003}
			\max\{1-r(s_0,s_1), V_{s_1}(\varphi)\} \leq 1-r(s,s_0)
		\end{equation}
		From transitivity we have $1-r(s,s_1) \leq \max\{1-r(s,s_0), 1-r(s_0, s_1)\}$,
		which from \ref{eq_002} yields 
		\begin{equation}\label{eq_004}
			1-r(s,s_1) \geq  1-r(s, s_0).	
		\end{equation}
		By definition, we have $V_s(B \varphi) \leq \max\{1-r(s,s_1), V_{s_1}(\varphi)\}$, and so from \ref{eq_003} and \ref{eq_004} we have $V_s(B \varphi) \leq 1-r(s,s_0) = V_s(B B \varphi)$.
				
	The other cases in which the value of $B$ obtains in the infimum will follow similar to above disscused cases.
		
	\begin{figure}\centering
		\resizebox*{3cm}{!}{
			\begin{tikzpicture}
				[->,shorten >=1pt,auto, thick, main node/.style={circle,fill=yellow!10,draw, font=\sffamily\Large\bfseries,minimum size=10mm}]
				\node[main node] (v2) at (0.5,1.5) {$s_0$};
				\node[main node] (v1) at (-1,-0.5) {$s_1$};
				\node[main node] (v3) at (2,-0.5) {$s$};
				\draw  (v2) edge (v1);
				\draw  (v3) edge (v2);
				\draw  (v1) edge (v3);			
			\end{tikzpicture}
		}
		\caption{A diagram of possible relation between the state $s_0$ and other states }\label{fig_trans_prop.}
	\end{figure}
			
	 (iii) Suppose that $r$ is reflexive and $s\in S$ is an arbitrary state. We show that $V_s(B \varphi)\leq V_s(\varphi)$. Recall that from the definition we have 
			$V_s(B \varphi) = \inf_{s' \in S}\max\{1-r(s,s'), V_{s'}(\varphi)\}$.
			Since $s\in S$, we have
			$$\max\{1-r(s,s), V_{s'}(\varphi)\} \geq \inf_{s'\in S}\max\{1-r(s,s'), V_{s'}(\varphi)\}.$$
			From reflexivity $1-r(s,s) = 0$, and so 
			$$\max\{0, V_{s}(\varphi)\} \geq \inf_{s'\in S}\max\{1-r(s,s'), V_{s'}(\varphi)\},$$
			which means that
			$V_{s}(\varphi) \geq V_s(B \varphi)$.
	\end{proof}
\begin{remark} \label{R2}
Unless otherwise specified, we will use the symbol \textbf{A} to represent any of the systems $\textbf{K}_\textbf{F}$, $\textbf{B}_\textbf{F}$, or $\textbf{T}_\textbf{F}$ in the following.
\end{remark}

\begin{definition}\label{def:derivation}
	A formula $\varphi$ is derived from a set of formulas $\Gamma$ in the axiomatic system \textbf{A}, denoted by $\Gamma\vdash_\textbf{A}\varphi$, if there exists a sequence of formulas $\varphi_1, \cdots, \varphi_n$ such that $\varphi_n = \varphi$, and for all $i\leq n$, each $\varphi_i$ is either an element of $\Gamma$, an axiom of \textbf{A}, or derived using rules \textbf{(R1)} and \textbf{(R2)} with the restriction that \textbf{(R2)} cannot be applied to non-valid formulas.
	\\
	In the case where the system \textbf{A} is $\textbf{K}_\textbf{F}$, we simply write $\Gamma\vdash\varphi$.

	\end{definition}

	\begin{definition}\label{consistency} Let $\varphi, \varphi_1,\ldots,\varphi_n$ be formulas in EGL-language.
		\begin{enumerate}
			\item $\varphi$ is \textbf{A}-consistent if $\nvdash_{\textbf{A}}\neg\varphi$,
			\item A finite set $\{\varphi_1,\ldots,\varphi_n\}$ is \textbf{A}-consistent if  $\varphi_1\wedge\ldots\wedge\varphi_n$ is \textbf{A}-consistent,
			\item An infinite set $\Gamma$ of formulas is \textbf{A}-consistent if any finite subset of $\Gamma$ is \textbf{A}-consistent,
			\item A formula or a set of formulas is called \textbf{A}-inconsistent if it is not \textbf{A}-consistent,
			\item A set $\Gamma$ of formulas is maximal \textbf{A}-consistent if:
			\begin{enumerate}
				\item $\Gamma$ is \textbf{A}-consistent,
				\item $\Gamma\cup\{\psi\}$ is \textbf{A}-inconsistent for any formula $\psi\notin\Gamma$.
			\end{enumerate}
		\end{enumerate}

	\end{definition}

		In cases where there is no ambiguity, we use the terms ``consistent" and ``inconsistent" instead of \textbf{A}-consistent and \textbf{A}-inconsistent, respectively.
	\begin{lemma}\label{maximallyConsistent}

		(1) Every consistent set of formulas can be extended to a maximal consistent set. (2) If $\Gamma$ is a maximal consistent set of formulas, then the following statements hold for all formulas $\varphi$, $\psi$:\\
		\hspace*{0.5cm}(a) either $\varphi\in\Gamma$ or $\neg\varphi\in\Gamma$,\\
		\hspace*{0.5cm}(b) $\varphi\wedge\psi\in\Gamma\ \Leftrightarrow\ \varphi\in\Gamma\ \text{and}\ \psi\in\Gamma$,\\
		\hspace*{0.5cm}(c) if $\varphi, ~\varphi\rightarrow\psi\in\Gamma$ then $\psi\in\Gamma$,\\
		\hspace*{0.5cm}(d) $\Gamma$ is closed under deduction, i.e. if $\Gamma\vdash\varphi$ then $\varphi\in\Gamma$.
	\end{lemma}

	\begin{proof}
Despite the fact that values of formulas range over  the interval $[0,1]$, the proof follows a similar approach to that of Lemma 1.4.3 in \cite{Meyer}.
	\end{proof}

	\section{Soundness and semi-completeness}
In this section, we establish the soundness and semi-completeness of the axiomatic systems $\textbf{K}_\textbf{F}$, $\textbf{B}_\textbf{F}$, and $\textbf{T}_\textbf{F}$ with respect to the corresponding classes of models. First, we provide some definitions:
	\begin{definition}
	Let $\mathcal{M}$ be a class of models. A  system \textbf{A} is called:
	\begin{itemize}
		
		\item[$\bullet$] \textit{sound} with respect to $\mathcal{M}$: if $\vdash_{\textbf{A}}\varphi$ then $\mathcal{M}\vDash\varphi$,
		\item[$\bullet$]\textit{complete} with respect to $\mathcal{M}$: if $\mathcal{M}\vDash\varphi$ then $\vdash_{\textbf{A}}\varphi$,
		\item[$\bullet$] \textit{semi-complete} with respect to $\mathcal{M}$: if $\mathcal{M}\vDash\varphi$ then $\vdash_{\textbf{A}}\neg\neg\varphi$.
	\end{itemize}
	\end{definition}

	\begin{lemma}\label{rule}
		The inference rules $\textbf{(R1)}$ and $\textbf{(R2)}$ are sound,  i.e. if the premises of $\textbf{(R1)}$ and $\textbf{(R2)}$ are valid, then their conclusions are also valid.  \vspace{-0.5cm} 
	\end{lemma}
	\begin{proof}
		The proof is obvious.
	\end{proof}

	\begin{theorem}\label{soundness}\textbf{(Soundness)}
		\begin{itemize}[nolistsep]
			\item[(1)] $\textbf{K}_\textbf{F}$ is sound  with respect to the class of all models.
			\item[(2)] $\textbf{B}_\textbf{F}$ is sound  with respect to the class of all  serial and transitive models.
			\item[(3)] $\textbf{T}_\textbf{F}$ is sound with respect to the class of all  serial, transitive and reflexive models.
		\end{itemize}
	\end{theorem}

	\begin{proof}
		The result can be obtained through a straightforward application of Proposition \ref{sound}, Lemma \ref{A4}, and Lemma \ref{rule}.
	\end{proof}
		
	Hereafter, we will adopt the following abbreviations:
	\begin{itemize}
		\item ``$\exists M \,\exists s$" is an abbreviation for ``there exists a model $M=(S, r_{} ,\pi)$ and there is some $s\in S$"
		\item ``$\forall M \, \forall s$" is an abbreviation for ``for all model $M=(S, r_{} ,\pi)$ and for all states $s\in S$"
	\end{itemize}

	Furthermore, to emphasize that we are referring to a specific model $M$, we will employ superscripts such as $S^M$, $\pi^M$, $r^M$, and $V_s^M$.

	\vskip 0.1cm
	\begin{lemma}(\textbf{Model Existence Lemma})\label{model_existent_lemma} 
		Let $\varphi$ be a  formula. Then the following statements are equivalent:
		\begin{enumerate}

			\item If $\vDash\varphi$, then $\vdash\neg\neg\varphi$,
			\item  if $\nvdash\neg\neg\varphi$ then there is a model $M=(S, r ,\pi)$ and a state $s\in S$ such that $(M,s)\nvDash \varphi$,
			\item  if $\neg\neg\varphi$ is consistent then it is satisfiable, i.e. there is a model  $M=(S, r ,\pi)$ and a state $s\in S$ such that $(M,s)\vDash \neg\neg\varphi$.
		\end{enumerate}
	\vspace{-0.5cm}
	\end{lemma}
		
	\begin{proof}
		Obviously (1) and (2) are  equivalent. \\
		(2)$\Rightarrow$(3): We restate  statements (2) and (3) as follows, respectively:
		$$\nvdash\neg\neg\varphi\ \Longrightarrow\ \exists M\ \exists s\ V_s(\varphi)\neq 1, $$
		$$\nvdash\neg\varphi\ \Longrightarrow\ \exists M\ \exists s\ V_s(\neg\neg\varphi)=1. $$
		Assume that (2) holds, then
		$$\nvdash\neg\neg\varphi\ \Longrightarrow\ \exists M\ \exists s\ V_s(\neg\varphi)=\left\{
		\begin{array}{lr}
			0& \;\;0<V_s(\varphi)<1\\
			1& \;\;V_s(\varphi)=0
		\end{array}\right..
	 $$
 	replacing $\varphi$ by $\neg\varphi$, we have
	$$\nvdash\neg\neg\neg\varphi\ \Longrightarrow\ \exists M\ \exists s\ V_s(\neg\neg\varphi)=\left\{
	\begin{array}{lr}

		0& \;\;0<V_s(\neg\varphi)<1\\
		1& \;\;V_s(\neg\varphi)=0
	\end{array}\right..$$
		However, since $\neg\varphi$ takes only crisp values, the case  $0<V_s(\neg\varphi)<1$ never happens,
		therefore,
		$$\nvdash \neg\varphi\ \Longrightarrow\ \exists M\ \exists s\ V_s(\neg\neg\varphi)=1.\vspace*{-0.1cm}$$
		(3) $\Rightarrow$(2):
		Assume that (3) holds. Then
		$\nvdash \neg\varphi$ implies $\exists M\ \exists s\ V_s(\neg\varphi)=0$, and
		if we replace $\varphi$ by $\neg\varphi$, we have
		$$\nvdash\neg\neg\varphi\ \Longrightarrow\ \exists M\ \exists s\ V_s(\neg\neg\varphi)=0.$$
		Therefore,
		$\nvdash\neg\neg\varphi$ implies that $\exists M\ \exists s\ V_s(\varphi)\neq 1$,
		which completes the proof.
	\end{proof}

\begin{lemma}\label{lem:consistency of axiom 6}
	$\Gamma_0 = \{\neg\neg B \nn \varphi \rightarrow \nn B \varphi \mid \varphi \text{ is a formula}\}$ is a consistent set of formulas.
\end{lemma}
\begin{proof}
	For the sake of contradiction assume $\Gamma_0$ is inconsistent and so there  are $\varphi_1,\cdots, \varphi_n \in \Gamma_0$ such that $\vdash \neg (F(\varphi_1)\wedge \cdots\wedge F(\varphi_n))$ in which  $F(\varphi_i) = \nn B \nn \varphi_i\rightarrow \nn B\varphi_i$. Using soundness we have $\vDash \neg (F(\varphi_1)\wedge \cdots\wedge F(\varphi_n)),$ which means that $\forall M\; \forall s \; V_s^{M}(\neg (F(\varphi_1)\wedge \cdots\wedge F(\varphi_n)))=1.$ Hence $\forall M \;\forall s\; V_s^M(F(\varphi_1)\wedge \cdots\wedge F(\varphi_n))=0$, and so $\forall M\; \forall s$ there is a  $ \varphi_i$
	such that $ V_s^M(F(\varphi_i)) = 0$. But this is not true since by Proposition \ref{axim6isvalid}, if $M^f$ is a finite model, then for all $s\in S^{M^f}$ and all $ \varphi_i \,(1\leq i\leq n)$ we have
	  $V_s^{M^f}(F(\varphi_i))=1,$  which is a contradiction.
\end{proof}
\begin{remark}
	Although in section \ref{section:finite_model_property} we will see that $\nn B \nn \varphi \rightarrow \nn B \varphi$ is not valid, Lemma \ref{lem:consistency of axiom 6} shows that the set $\Gamma_0$ is a consistent set of formulas.
\end{remark}

		\begin{theorem}\label{weakcompleteness} \textbf{(Semi-Completeness)}
			$\textbf{K}_\textbf{F}$ is semi-complete with respect to the class of all models. \vspace{-0.3cm}
		\end{theorem}
		\begin{proof}
			According to the Model Existence Lemma, it suffices to demonstrate that for each formula $\varphi$, if $\neg\neg\varphi$ is consistent, then $\neg\neg\varphi$ is satisfiable.
			 To achieve this, we construct a semi-canonical model denoted as $M^c=(S^c, r^c_{} ,\pi^c)$ using $\Gamma_0$ defined in Lemma \ref{lem:consistency of axiom 6} as follows:\\
		$~~~\bullet\ S^c=\left \lbrace s_{\Theta}\mid \ \begin{array}{l}\Theta
		\text{ is a maximal consistent set of formulas, and } 
		\Gamma_0 \subseteq \Theta\\
		\end{array}\right \rbrace$\\
		$~~~\bullet\ r^c_{}(s_{\Theta},s_{\Psi})=\left\{
		\begin{array}{lr}
			1&\;\; \Theta/B\subseteq\Psi\\
			0&\;\; \text{otherwise}
		\end{array}\right.$, 
		where $\Theta/B=\{\neg\neg\varphi\mid\neg\neg B\varphi\in\Theta\}$,\\
		$~~~\bullet\ \pi^c(s_{\Theta},p)=\left\{
		\begin{array}{lr}
			1&\;\; \neg\neg p\in\Theta\\
			0&\;\; \neg\neg p\notin\Theta
		\end{array}\right.
		$,  where $p\in\mathcal{P}$.\\

		Naturally, $\pi^c$ can be extended to all formulas, and we will use the notation $V$ to represent its extension. It is important to note that both $\pi^c$ and, consequently, $V$ assign crisp values. Additionally, it should be noted that $S^c$ is well-defined by Lemma \ref{lem:consistency of axiom 6}.
		Actually the condition $\Gamma_0\subseteq \Theta$ in the definition of $S^c$ is just required for the last part of the proof.
		Now, it is sufficient to demonstrate that for each formula $\varphi$ and each $s_{\Theta}\in S^c$, the following statement holds:
		$$V_{s_{\Theta}}(\neg\neg\varphi)=1\ \Longleftrightarrow\ \neg\neg\varphi\in\Theta $$
		We prove it by induction on the structure of $\varphi$:\\\\
		\textbf{case 1}: $\varphi=p$, where $p\in\mathcal{P}$:
		$$V_{s_{\Theta}}(\neg\neg p)=1\ \Leftrightarrow\ V_{s_{\Theta}}(p)>0\ \Leftrightarrow \ V_{s_{\Theta}}(p)=1\ \Leftrightarrow\ \pi^c(s_{\Theta},p)=1\ \Leftrightarrow\ \neg\neg p\in\Theta$$
		\textbf{case 2}: $\varphi=\varphi_1\wedge\varphi_2$
		\begin{align*}
			&V_{s_{\Theta}}(\neg\neg (\varphi_1\wedge\varphi_2))=1\ \Leftrightarrow\ V_{s_{\Theta}}(\varphi_1\wedge\varphi_2)>0\ \Leftrightarrow \ V_{s_{\Theta}}(\varphi_1)>0, V_{s_{\Theta}}(\varphi_2)>0\ \Leftrightarrow\\
			&V_{s_{\Theta}}(\neg\neg\varphi_1)=1, V_{s_{\Theta}}(\neg\neg\varphi_2)=1\ \Leftrightarrow\ \neg\neg\varphi_1, \neg\neg\varphi_2\in\Theta\quad\text{(induction hypothesis)}\\
			&\Leftrightarrow\ \neg\neg\varphi_1\wedge\neg\neg\varphi_2\in\Theta\quad(\text{maximal consistency of}\ \Theta)\\
			&\Leftrightarrow\ \neg\neg(\varphi_1\wedge\varphi_2)\in\Theta\quad(\text{(GT4) and maximal consistency of}\ \Theta)
		\end{align*}
		\\
		\textbf{case 3}: $\varphi=\varphi_1\rightarrow\varphi_2$
		
		($\Leftarrow$) Let $\neg\neg(\varphi_1\rightarrow\varphi_2)\in\Theta$. Then using (GT5) and maximal consistency of $\Theta$ we have
		\begin{equation}\label{2}
			\neg\neg\varphi_1\rightarrow\neg\neg\varphi_2\in\Theta
		\end{equation}

		Assume that $V_{s_{\Theta}}(\neg\neg(\varphi_1\rightarrow\varphi_2))=0$. Then, $V_{s_{\Theta}}(\varphi_1\rightarrow\varphi_2)=0$ and so it is demonstrated that $V_{s_{\Theta}}(\varphi_1)=1$ and $V_{s_{\Theta}}(\varphi_2)=0$. Hence, $V_{s_{\Theta}}(\neg\neg\varphi_1)=1$ and $V_{s_{\Theta}}(\neg\neg\varphi_2)=0$. By induction hypothesis $\neg\neg\varphi_1\in\Theta$, then by applying Lemma \ref{maximallyConsistent}, the statement (\ref{2}) and the maximal consistency of  ${\Theta}$, conclude that  $\neg\neg\varphi_2\in\Theta$. Consequently, by induction hypothesis  $V_{s_{\Theta}}(\neg\neg\varphi_2)=1$, contradicting the assumption.
		
		($\Rightarrow$) Assume that $V_{s_{\Theta}}(\neg\neg(\varphi_1\rightarrow\varphi_2))=1$, then $V_{s_{\Theta}}(\varphi_1\rightarrow\varphi_2)=1$. Thus it is obtained that either $V_{s_{\Theta}}(\varphi_2)=1$ or $V_{s_{\Theta}}(\varphi_1)=V_{s_{\Theta}}(\varphi_2)=0$. If $V_{s_{\Theta}}(\varphi_2)=1$, it is derived that $V_{s_{\Theta}}(\neg\neg\varphi_2)=1$ and then by induction hypothesis $\neg\neg\varphi_2\in\Theta$. Since $\neg\neg\varphi_2\rightarrow(\neg\neg\varphi_1\rightarrow\neg\neg\varphi_2)$ is an instance of (GT1), then maximal consistency of $\Theta$ results that  $\neg\neg\varphi_1\rightarrow\neg\neg\varphi_2\in\Theta$. Therefore, by (GT5) we obtain that $\neg\neg(\varphi_1\rightarrow\varphi_2)\in\Theta$. 
		In other case, if $V_{s_{\Theta}}(\varphi_1)=V_{s_{\Theta}}(\varphi_2)=0$, we have $V_{s_{\Theta}}(\neg\neg\varphi_1)=V_{s_{\Theta}}(\neg\neg\varphi_2)=0$. Thus by induction hypothesis $\neg\neg\varphi_1, \neg\neg\varphi_2\notin\Theta$ and so by maximality  $\neg\varphi_1, \neg\varphi_2\in\Theta$. Then we have $\neg\varphi_2\rightarrow\neg\varphi_1\in\Theta$, using (GT1) and  $\neg\neg\varphi_1\rightarrow\neg\neg\varphi_2\in\Theta$, by applying (GT2). Finally $\neg\neg(\varphi_1\rightarrow\varphi_2)\in\Theta$ follows from (GT5).
		\\
		\textbf{case 4}: $\varphi=B\psi$
		
		($\Leftarrow$) Assume that $\neg\neg B\psi\in\Theta$, then $\neg\neg\psi\in\Theta/B$. Let $\Psi$ be an arbitrary maximal consistent set of formulas and $s_{\Psi}\in S^c$. Then we have:
		\begin{align*}
			r^c(s_{\Theta},s_{\Psi})=1\ \Rightarrow\ \neg\neg\psi\in\Theta/B\subseteq\Psi\ &\Rightarrow\ V_{s_{\Psi}}(\neg\neg\psi)=1\quad\text{(induction hypothesis)}\\
			& 
			\Rightarrow\ V_{s_{\Psi}}(\psi)=1
		\end{align*}

			Consequently, 
			$V_{s_{\Theta}}(B\psi)=\inf_{s\in S^c}\max\big\{1-r^c(s_{\Theta},s),  V_{s}(\psi)\big\}=1$, so $V_{s_{\Theta}}(\neg\neg B\psi)=1$.
			
			($\Rightarrow$) Assume that $V_{s_{\Theta}}(\neg\neg B\psi)=1$.\\
		\textsc{Claim}: $\Theta/B\cup\{\neg\psi\}$ is inconsistent.\\
		\textit{Proof of the Claim.} Suppose that $\Theta/B\cup\{\neg\psi\}$ is consistent. Then by some modification in the proof of lemma \ref{maximallyConsistent}, there exists a maximal consistent extension $\Psi$ such that $\Theta/B\subseteq\Theta/B\cup\{\neg\psi\}\subseteq\Psi$. Thus, $r^c(s_{\Theta},s_{\Psi})=1$ and $\neg\psi\in\Psi$. 
		By maximality of $\Psi$, we have $\nn\psi\notin\Psi$ and by induction hypothesis $V_{s_{\Psi}}(\nn\psi)=0$, and then $V_{s_{\Psi}}(\psi)=0$. Therefore,
		$$V_{s_{\Theta}}(B\psi)=\dis\inf_{s_{\Psi}\in S^c}\max\big\{1-r^c(s_{\Theta},s_{\Psi}),  V_{s_{\Psi}}(\psi)\big\}=0 $$
		and it is concluded that $V_{s_{\Theta}}(\nn B\psi)=0$, which contradicts the assumption. $\blacksquare_\textit{Claim}$

	Consequently, it follows that there is an inconsistent finite subset $\Delta$ of $\Theta/B\cup\{\neg\psi\}$, and assume that $\Delta=\{\nn\varphi_1,\ldots,\nn\varphi_k,\neg\psi\}$. Note that without loss of generality we can suppose that $\Delta$ contains $\neg\psi$. 
	Hence
				$$\vdash\neg(\nn\varphi_1\wedge\ldots\wedge\nn\varphi_k\wedge\neg\psi). $$
	By (GT4), (GT7) and transitivity rule it can be obtained that  
	$$\vdash\neg(\varphi_1\wedge\ldots\wedge\varphi_k\wedge\neg\psi)$$		
	then 
			$\vdash\varphi_1\rightarrow(\varphi_2\rightarrow(\ldots(\varphi_k\rightarrow\nn\psi)\ldots))$ by using (GT3) and (G2). Consequently by applying \textbf{(R2)} we obtain $\vdash B(\varphi_1\rightarrow\xi)$, where   $\xi$ is $\varphi_2\rightarrow(\ldots(\varphi_k\rightarrow\nn\psi)\ldots))$. Thus  $\vdash\nn B(\varphi_1\rightarrow\xi)$, by (GT6).  Since $\Theta$ is a maximal consistent set, then by Lemma \ref{maximallyConsistent} the following statement holds
			\begin{equation}\label{3}
				\nn B(\varphi_1\rightarrow\xi)\in\Theta
			\end{equation}

			Also from $\nn\varphi_1,\ldots,\nn\varphi_k\in\Theta/B$, it is demonstrated that
			\begin{equation}\label{4}
				\nn B\varphi_1,\ldots,\nn B\varphi_k\in\Theta
			\end{equation}

			Thus, we have
			\begin{align*}
				&\vdash (B\varphi_1\wedge B(\varphi_1\rightarrow\xi))\rightarrow B\xi &&(\text{by \textbf{(K)}})\\
				& \vdash B\varphi_1\rightarrow (B(\varphi_1\rightarrow\xi)\rightarrow B\xi)&&(\text{by (G2)})\\
				& \vdash \nn(B\varphi_1\rightarrow (B(\varphi_1\rightarrow\xi)\rightarrow B\xi))&&(\text{by (GT6)})\\
				& \vdash \nn B\varphi_1\rightarrow \nn (B(\varphi_1\rightarrow\xi)\rightarrow B\xi)&&(\text{by (GT5)})\\
				& \vdash \nn B\varphi_1\rightarrow (\nn B(\varphi_1\rightarrow\xi)\rightarrow \nn B\xi)&&(\text{by (GT5)})
			\end{align*}

			Therefore,

			\begin{equation}\label{5}
				\nn B\varphi_1\rightarrow (\nn B(\varphi_1\rightarrow\xi)\rightarrow \nn B\xi)\in\Theta.
			\end{equation}

			By (\ref{4}), (\ref{5}) and  maximal consistency of $\Theta$ it follows that:
			\begin{equation}\label{6}
				\nn B(\varphi_1\rightarrow\xi)\rightarrow \nn B\xi\in\Theta
			\end{equation}

				Furthermore, utilizing (\ref{3}), (\ref{6}), and the maximal consistency of $\Theta$, we can deduce that $\nn B\xi\in\Theta$. Repeatedly applying the aforementioned process leads to $\nn B\nn\psi\in\Theta$. Since $\neg \neg B \neg \neg \psi \rightarrow \neg \neg B \psi \in \Theta$ based on the definition of $\Theta \in S^c$, it follows that $\nn B\psi\in\Theta$, which completes the proof.
			
		\end{proof}

			\begin{remark}
			
					As the proof of Theorem \ref{weakcompleteness} exclusively relied on models with crisp values and accessibility relations, it follows that $\textbf{K}_\textbf{F}$ is semi-complete with respect to the class of classical models.
				
			\end{remark}
			
		\begin{corollary}
			\label{semi-completeness-corollary}
			$\quad $
			\begin{itemize}[nolistsep]
				\item[(i)] $\textbf{B}_\textbf{F}$ is semi-complete with respect to the class of all  serial and transitive models.
				\item[(ii)] $\textbf{T}_\textbf{F}$ is semi-complete with respect to the class of all  serial, transitive and reflexive models.
			\end{itemize}

		\end{corollary}
		\begin{proof}
			
			Consider an axiomatic system \textbf{A} that includes $\textbf{K}_\textbf{F}$. According to Lemma \ref{A4}, the axiom $\textbf{(D)}$ corresponds to the seriality of models. Therefore, if $\textbf{(D)}$ is a member of \textbf{A}, we can construct a semi-canonical model for \textbf{A} similar to the one used in the proof of Theorem \ref{weakcompleteness}, which possesses the serial property. Additionally, we demonstrate that if the axioms $\textbf{(4)}$ and $\textbf{(T)}$ are elements of \textbf{A}, the constructed semi-canonical model will be transitive and reflexive, respectively. In the subsequent discussions, let $\Gamma_0$ denote the set introduced in Lemma \ref{lem:consistency of axiom 6}.

			(I) 				Suppose $\textbf{(4)}\in \textbf{A}$. We demonstrate that the model $M^c$ introduced in Theorem \ref{weakcompleteness} is transitive. To prove this, we consider three maximal consistent sets of formulas, namely $\Theta_1$, $\Theta_2$, and $\Theta_3$, each containing $\Gamma_0$. We must show that
			
			$$r^c(s_{\Theta_1},s_{\Theta_3})\geq \min\{r^c(s_{\Theta_1},s_{\Theta_2}), r^c(s_{\Theta_2},s_{\Theta_3})\}.$$
			Since the relations in the semi-canonical model have crisp values, we only need to establish $r^c(s_{\Theta_1},s_{\Theta_3})=1$ based on the assumptions $r^c(s_{\Theta_1},s_{\Theta_2})=r^c(s_{\Theta_2},s_{\Theta_3})=1$. Consequently, using these assumptions and the definition of $r^c$, we obtain the following:
			
			\begin{equation}\label{rc1}
				\{\nn\varphi\mid\nn B\varphi\in\Theta_1\}\subseteq\Theta_2,
			\end{equation}
			\begin{equation}\label{rc2}
				\{\nn\varphi\mid\nn B\varphi\in\Theta_2\}\subseteq\Theta_3.
			\end{equation}

		Now suppose that $\nn B\varphi\in\Theta_1$.
				Given the maximality and consistency of $\Theta_1$, we have $B \varphi \in \Theta_1$. Then according to the axiom $\textbf{(4)}$, we obtain $ B B \varphi \in \Theta_1$, and consequently, $\nn B B \varphi \in \Theta_1$. Finally, $\nn B \varphi\in \Theta_2$ is derived by (\ref{rc1}), and $\nn\varphi\in\Theta_3$ is deduced by employing (\ref{rc2}).
				Hence, we can conclude that
				$\{\nn\varphi\mid\nn B\varphi\in\Theta_1\}\subseteq\Theta_3$, which implies $r^c(s_{\Theta_1},s_{\Theta_3})=1$.
			
			(II) 
				Suppose $\textbf{(T)}\in S$. We will demonstrate that $M^c$ is a reflexive model. Let $\Theta$ be a set of maximal consistent formulas containing $\Gamma_0$. By applying (GT2) twice, we obtain $\nn B\varphi\to\nn\varphi\in\Theta$ from $B\varphi\to\varphi\in\Theta$. Consequently, $\{\nn\varphi\mid\nn B\varphi\in\Theta\}\subseteq\Theta$, implying that $r^c(s_{\Theta},s_{\Theta})=1$. This completes the proof.
		\end{proof}
		\begin{corollary}
		\label{ExistModel}
		If $\Theta$ is a maximal and \textbf{A}-consistent set containing $\Gamma_0$ introduced in Lemma \ref{lem:consistency of axiom 6}, then for all $\psi \in \Theta$ we have
		\begin{equation}
			\exists M\; \exists s\; \; 0< V_s(\psi) \leq 1.
		\end{equation}
		\end{corollary}
		\begin{proof}
			
				Consider a maximal and \textbf{A}-consistent set $\Theta$. Let us assume, for the sake of contradiction,  that $\psi \in \Theta$ and for all models $M=(S, r, \pi)$ and all states $s\in S$, $V_s(\psi) = 0$. Consequently, $\vDash \neg \psi$. According to the Semi-Completeness Theorem \ref{weakcompleteness}, we have $\vdash_\textbf{A} \neg \neg \neg \psi$, and by utilizing (GT7), we deduce $\vdash_\textbf{A} \neg \psi$. Since $\Theta$ is maximal and \textbf{A}-consistent, we would then have $\neg \psi \in \Theta$. However, this is not possible since $\psi \in \Theta$ and $\Theta$ is \textbf{A}-consistent.
			
		\end{proof}

	\section{Completeness}											
	
			In this section, we prove the completeness of $\textbf{K}_\textbf{F}$ using a new approach since we can not apply the traditional Lindenbaum's method.
			  Furthermore, we establish that $\textbf{B}_\textbf{F}$ is complete with respect to the class of all serial and transitive models, and $\textbf{T}_\textbf{F}$ is complete with respect to the class of all serial, transitive, and reflexive models.
	
	\begin{definition}

			Let $\varphi$ and $\psi$ be two formulas. We say that $\varphi$ is \textit{dependent on the zero values of} $\psi$ if, for all models $M$ and all states $s$ in $M$, whenever $V_s(\psi)=0$, then $V_s(\varphi)=0$.
			
			We use the notation $\varphi \equiv \psi$ to indicate that two formulas $\varphi$ and $\psi$ are semantically equivalent, meaning that for all models $M$ and all states $s$, $V_s(\varphi) = V_s(\psi)$.
		
	\end{definition}

	\begin{lemma} \label{LEMMA_01}		
		Let $\psi$ and $\chi$ be formulas. If $\psi$ is dependent on the zero values of $\chi$, then  $\psi \equiv \neg \neg \chi \wedge \psi$.
	\end{lemma}

	\begin{proof}
		
			Suppose $M$ is an arbitrary model, and $s$ is one of its states. We consider two cases.
			First, assume $V_s(\chi) = 0$. Since $\psi$ is dependent on the zero values of $\chi$, we have $V_s(\psi) = 0$. Consequently, $V_s(\psi) = V_s(\neg \neg \chi \wedge \psi) = 0$.
			Now, let's consider the case where $V_s(\chi) > 0$. In this case, we have $V_s(\neg \neg \chi) = 1$, and similarly, we have $V_s(\psi) = V_s(\neg \neg \chi \wedge \psi)$.

	\end{proof}

One traditional method for proving the completeness theorem is through Lindenbaum's approach. In this approach, It is assumed that a formula $\varphi$ is not provable, i.e., $\nvdash \varphi$. Then, it is shown that for a maximal and consistent set $\Theta$ containing $\neg\varphi$, there exists a Canonical model $M$ and a state $s$ such that $V_s^{M}(\neg \varphi)>0$, and this leads to a contradiction with the assumption $\vDash \varphi$.

However, in our logic, there are certain restrictions that prevent us from using the traditional Lindenbaum's approach.
We encounter formulas $\varphi$ for which $\nvDash \varphi$ and hence $\nvdash \varphi$, and they possess the following properties:
\begin{equation}\label{eq_030}
\forall M, \, \forall s \quad V_s^{M}(\varphi)>0,\quad \text{ and } \quad\exists M'\, \exists s' \quad V_{s'}^M(\varphi)<1.	
\end{equation}
For such formulas, we have $\neg \varphi\equiv \perp$, and thus they cannot be contained in any maximal consistent set of formulas. Consequently, we propose a new approach to prove the completeness theorem. Specifically, we categorize these types of formulas and employ various lemmas to establish the completeness theorem.

	\begin{definition}
		Let $\varphi$, $\psi$, $\chi$, and $\xi$ be formulas. We define the structural equivalence relation between implications as follows: $\varphi \rightarrow \psi$ is said to be structurally equivalent to $\chi \rightarrow \xi$, denoted as $\varphi \rightarrow \psi \approxeq \chi \rightarrow \xi$, if $\varphi \rightarrow \psi \equiv \chi \rightarrow \xi$ and $\psi \equiv \xi$.
	
	\end{definition}
	
	\begin{definition} \label{def_4_sets}
		Let $\mathcal{T} = \{\varphi\, \mid\, \vdash_{\textbf{A}}\varphi\}, $	
	 $\varphi_t \in \mathcal{T}$ and $\varphi, \psi$ be formulas. Suppose there exists a model $M$ and a state $s$ such that $0 < V_s(\psi) < 1$ and $V_s(\psi) < V_s(\varphi)$. Consider the following grammar:
		$$\begin{array}{llc}
				\varphi_{bad} &::=&  (\neg \neg \psi \wedge \varphi) \rightarrow \psi \,\mid\, \neg \neg \psi \rightarrow \psi \\
				\varphi_e &::= & \chi \rightarrow \chi' \,\mid\, (\neg \neg \chi'\wedge\varphi) \rightarrow \chi'\,\mid \, \neg \neg \chi'\rightarrow \chi'\\
				\varphi_o &::= & \chi' \wedge \chi \\
				\chi & ::= & \varphi_{bad}\,\mid\, \varphi_e \,\mid\, \varphi_o \mid\,\varphi_t\\
				\chi' &::=& \varphi_{bad}\,\mid\, \varphi_e \,\mid\, \varphi_o \\
\end{array}$$
		In the generation of $\varphi_e$, we have two restrictions: $\chi' \not \equiv \chi$ and $\neg \neg \chi' \wedge \varphi \not \equiv \perp$. Additionally, since $\xi' \wedge \xi$ and $\xi \wedge \xi'$ are semantically equivalent, we consider only one of them in $\varphi_o$. We define $\mathcal{E}$ as the set containing all formulas structurally equivalent to a $\varphi_{bad}$-formula or a $\varphi_e$-formula, and $\mathcal{O}$ as the set of all $\varphi_o$-formulas.
						\end{definition}
%

		\begin{remark} $\mathcal{T}$, $\mathcal{E}$ and $\mathcal{O}$ are disjoint sets of formulas.
	\end{remark}
\begin{proposition}\label{rem:EuO_arent_valid}
	For each formula $\varphi \in \mathcal{E}\cup\mathcal{O}$, there is a model $M$ and a state $s$ such that $(M,s)\nvDash \varphi$.
\end{proposition}
\begin{proof}

		We will demonstrate the argument for one possible form of the formula $\varphi$ in Definition \ref{def_4_sets}, but the same reasoning can be applied to other forms as well. Let $\varphi = (\neg\neg \psi \wedge \varphi') \rightarrow \psi$. Consider a model $M$ and a $s \in S^M$ such that $0 < V_s^M(\psi) < 1$ and $V_s^M(\psi) < V_s^M(\varphi')$. Based on these facts, we have $V_s^M(\neg \neg \psi) = 1$, which implies $V_s^M(\neg \neg \psi \wedge \varphi') = V_s^M(\varphi') > V_s^M(\psi)$. Therefore, $V_s^M(\varphi) = V_s^M(\psi)$, indicating that $(M,s) \nvDash \varphi$.

\end{proof}
\begin{proposition}\label{rem:existence_of_reflexvie_model_for_E_and_O}
	For each formula $\varphi\in \mathcal{E}\cup\mathcal{O},$ there is a reflexive model $M$ and a state $s\in S^M$ such that $(M,s)\nvDash \varphi$.
\end{proposition}
\begin{proof}
	Let $\varphi \in \mathcal{E}\cup\mathcal{O}$. According to Proposition \ref{rem:EuO_arent_valid}, there exists a model $M$ and a state $s' \in S^M$ such that $(M,s') \nvDash \varphi$. We define $M' = (S^M, r', \pi^M)$, where for all $s, s' \in S^M$ with $s \neq s'$, we have $r(s,s') = r'(s,s')$, and $r'(s,s) = 1$. 
	We can easily prove by induction that $(M',s') \nvDash \varphi$. In fact, we only need to consider the case $\varphi = B\psi$.  By using the definition of $V_{s'}^{M'}(B\psi)$, we have 
	$$\begin{array}{ll}
		V_{s'}^{M'}(B\psi) & = \inf_{s''\in S}\max\{1-r^{M'}(s',s''), V_{s''}(\psi) \}\\
		& \leq \inf_{s''\in S}\max\{1-r^{M}(s',s''), V_{s''}(\psi) \} \\
		& = V_{s'}^{M}(B\psi).
	\end{array}  $$
\end{proof}

	\begin{lemma} \label{lemma:notinTau_leads_to_have_model}

		Let $\varphi$ be a formula, and $\mathbb{N}$ be the set of natural numbers. If there exists a reflexive model $M$ and a state $s\in S^M$ such that $(M,s)\nvDash \varphi$ (and hence $\nvDash \varphi$), then there exists a reflexive model $M'=(\mathbb{N}, r, \pi)$ and an element $n\in\mathbb{N}$ such that for infinitely many $m\in \mathbb{N}$, we have: $$\max\{1-r^{M'}(n,m), V_m^{M'}(\varphi)\}\leq \frac{1}{m}.$$
		\end{lemma}
	\begin{proof}
	The proof can be found in the appendix.
	\end{proof}

	\begin{theorem}\label{NotEqualToPerp}
		Let $\varphi$ be a formula. Then we have
		$$\neg \varphi \equiv \perp \quad \text{ if and only if } \quad \varphi \in \mathcal{T} \cup \mathcal{E} \cup \mathcal{O}.$$
	\end{theorem}
	\begin{proof}
		{\bf$(\Rightarrow)$} 
			First, suppose that $\neg \varphi \equiv \perp$. We will use induction on the complexity of $\varphi$ to demonstrate that $\varphi \in \mathcal{T} \cup \mathcal{E} \cup \mathcal{O}$. The base step is trivially correct since it relies on the denial of the antecedent. Now, let's consider the induction step where $\alpha$ and $\beta$ are formulas:

		\textbf{case 1.} $\varphi = \neg \alpha$.
		
		$\vDash \neg \alpha$, since $\neg \neg \alpha \equiv \perp$.
 By Semi-Completeness theorem, we have $\vdash_\textbf{A} \neg \neg \neg \alpha$, and using $(GT7)$ we obtain $\vdash_\textbf{A} \neg \alpha$, and thereby $\varphi \in \mathcal{T}$.
 
		
		\textbf{case 2.} $\varphi = \alpha \wedge \beta$.
		
		Since $\neg (\alpha \wedge \beta) \equiv \perp$, we can conclude $\forall M\, \forall s\; 0<V_s(\alpha \wedge \beta)$. As a result, both $\alpha$ and $\beta$ satisfy the induction hypothesis and belong to $\mathcal{T} \cup \mathcal{E} \cup \mathcal{O}$. Now, let's demonstrate that $\alpha \wedge \beta \in \mathcal{T} \cup \mathcal{E} \cup \mathcal{O}$ by considering different cases.
		For instance, if both $\alpha$ and $\beta$ belong to $\mathcal{T}$, then by applying $(GT9)$ and utilizing the rule of Modus Ponens (MP), we can deduce that $\vdash_\textbf{A} \alpha \wedge\beta$. Consequently, $\alpha \wedge \beta \in \mathcal{T}$. In other cases, it can be established that $\alpha \wedge \beta \in \mathcal{O}$ based on the definition \ref{def_4_sets}.

		\textbf{case 3.} $\varphi = \alpha \rightarrow \beta$.

			If $\neg (\alpha \rightarrow \beta) \equiv \perp$, then it follows that $\forall M \; \forall s \; 0< V_s(\alpha \rightarrow \beta)$. This indicates that $\alpha$ is dependent on  zero values of $\beta$. By Lemma (\ref{LEMMA_01}), we have $\varphi \equiv (\neg \neg \beta \wedge \alpha)\rightarrow \beta$. Therefore, $\alpha \rightarrow \beta \approxeq (\neg \neg \beta \wedge \alpha)\rightarrow \beta$.
			If $\alpha \equiv \perp$, then $\alpha \rightarrow \beta \in \mathcal{T}$. On the other hand, if $\exists M \, \exists s \, V_s(\beta)<1$, then we can conclude that $\varphi \in \mathcal{E}$. In the case where $\forall M \, \forall s \, V_s(\beta)=1$, it is implied that $\neg \beta \equiv \bot$. By the induction hypothesis we have $\beta \in \mathcal{T} \cup \mathcal{E} \cup \mathcal{O}$. However, Proposition \ref{rem:EuO_arent_valid} implies that $\beta \notin \mathcal{E}\cup\mathcal{O}$. Hence, $\beta\in \mathcal{T}$, and we can derive $\vdash_{\textbf{A}}\beta$.
			Using  $\beta \rightarrow (\alpha \rightarrow \beta)$ which is a theorem in axiomatic systems $\textbf{K}_\textbf{F}$, $\textbf{B}_\textbf{F}$, and $\textbf{T}_\textbf{F}$, we can infer $\vdash_\textbf{A} \alpha \rightarrow \beta$, and so $\varphi \in \mathcal{T}$.

		\textbf{case 4.} $\varphi = B \alpha$. 
		
		Let $\neg (B \alpha) \equiv \perp$. Then
		\begin{equation} \label{eq_Bgeqzero}
			\forall M\; \forall s\;\;\;\; 0<V_s^M(B \alpha)\leq 1.
		\end{equation}

		 	We claim that $\forall M \; \forall s \; 0<V_s^M(\alpha)$. To prove this claim, we assume that if for a model $M'$ there exists a state $s'$ such that $V_{s'}^{M'}(\alpha)=0$, then there exists a model $M^*$ and a state $s^*$ such that $V_{s^*}^{M^*}(B \alpha)=0$ which this leads to a contradiction by equation (\ref{eq_Bgeqzero}). So let $M'$ be a model that contains a state $s'$, where
		 
		\begin{equation}\label{eq:00}
			V_{s'}^{M'}(\alpha)=0.
		\end{equation}.
		
		First suppose that $M'$ has more than one state. Then there exists $s'\neq s'' \in S^{M'}$ such that $V_{s''}^{M'}(B\alpha)>0$ by (\ref{eq_Bgeqzero}), and so  $r^{M'}(s'',s')\neq 1$, by the definition of $V_{s''}^{M'}(B\alpha)$.
		We define $M^*$ as follows:
		\vskip 0.2cm
		$S^{M^*} = S^{M'};$
		
		$\forall s\in S^{M^*} \forall p \in \mathcal{P} \quad \pi^{M^*}(s, p) = \pi^{M'}(s, p);$
		
		$r^{M^*}(s',s'') = r^{M'}(s', s''),\qquad r^{M^*}(s'',s') =1,$		
		
		$\forall t_0 , t_1 \in S^{M^*} \; \text{s.t. }\; t_0\neq s' \text{ or } t_1\neq s'' \qquad r^{M^*}(t_0,t_1)=r^{M'}(t_0,t_1).
		$
		\vskip 0.2cm
		The only difference between $M^*$ and $M'$ is $r^{M^*}(s'', s')$.
		We first prove that $V_{s'}^{M^*}(\alpha) = 0$.  For this, we use induction on the complexity of $\alpha$ to show that for infinitely many $t_m\in S^{M'}$  we have $V_{t_m}^{M^*}(\alpha) \leq V_{t_m}^{M'}(\alpha)$, specially $V_{s'}^{M^*}(\alpha) = V_{s'}^{M'}(\alpha)$ (Note that if $S^{M'}$ is finite, then we prove the inequality  for all $t_m\in S^{M'}$). We only check the case $\alpha = B \psi$, the other cases can be easily obtained.
		First, note that since the only difference between $M^*$ and $M'$ is $r^{M^*}(s'', s')$, we have $V_{s''}^{M^*}(B\psi) \leq V_{s''}^{M'}(B\psi)$ by definition, and it can be checked  for all other states $t_m$, using induction hypothesis.
		
		By (\ref{eq:00}) we have $V_{s'}^{M'}(B \psi) = 0$. So either there is a state $t\in S^{M'}$ in which $V_t^{M'} (\psi) = 0$ and $r^{M'}(s',t) = 1$, or there are infinitely many states $t_m\in S^{M'}$ such that $\max\{1-r(t,t_m), V_{t_m}^{M'}(\psi)\}\leq\frac{1}{m}$ (Note that the first case occurs when $S^{M'}$ is finite).
		In the former case, by induction hypothesis $V_{t_m}^{M^*}(\psi) \leq V_{t_m}^{M'}(\psi)$, then we have $V_t^{M^*}(\psi)=0$, and by definition of $M^*$ we have $r^{M^*}(s',t)=1$. 
		Moreover,

		$$V_{s'}^{M^*}(B \psi) \leq \max\{1-r^{M^*}(s',t), V_t^{M^*}(\psi)\} = 0.$$	
		Therefore, $V_{s''}^{M^*}(B \alpha) \leq \max\{1-r^{M^*}(s'',s'), V_{s'}^{M^*}(\alpha)\} = 0$ and so $V_{s''}^{M^*}(B \alpha)=0$ which is a contradiction.
		
		In latter case, by induction hypothesis we have $\max\{1-r(t,t_m), V_{t_m}^{M^*}(\psi)\}\leq \frac{1}{m}$, for infinitely many $t_m \in S^{M^*}$. So $V_{s'}^{M^*}(B\psi)=0$ and by a similar argument to the previous case we can achieve a contradiction.
		
		
		Now, suppose that $M'$ has a single state $s'$ . We extend it to a model $M''$ with two states as follows:
		
		$S^{M''} = \{s', t''\}$
		
		$r^{M''}(s',s') = r^{M'}(s',s'), \quad r^{M''}(s',t'') = 0,\quad r^{M''}(t'',s')=0.5$
		
		$(\forall p\in \mathcal{P}) \;\;\; \pi^{M''}(s',p) = \pi^{M'}(s',p)$.
		
		Note that the value of $r^{M''}(t'',t'')$ and $\pi(t'', p)$ for all $p\in\mathcal{P}$ are not important in our argument.  
		We show that $V_{s'}^{M'} (\alpha) = V_{s'}^{M''} (\alpha) $, by induction on the complexity of $\alpha$. 
		\begin{itemize}

			\item $\alpha = p$: It is obvious. 
			\item$\alpha = \neg \psi$:
			If $V_{s'}^{M'}(\neg \psi)=0$, then $V_{s'}^{M'}(\psi) >0$, so we have $V_{s'}^{M''}(\psi) >0$ by induction hypothesis. Thus $V_{s'}^{M''}(\neg \psi )= 0$. The case $V_{s'}^{M'}(\neg \psi) \neq 0$ has a similar argument.
			\item $\alpha = \psi \wedge \xi$:
				If $V_{s'}^{M' }(\psi \wedge \xi) = 0$, then at least one of $V_{s'}^{M'} (\psi)$ or $ V_{s'}^{M'} (\xi)$ is equal to 0. By induction, this implies that at least one of $V_{s'}^{M''} (\psi)$ or $ V_{s'}^{M''} (\xi)$ is equal to zero, confirming the validity of the statement. A similar argument can be applied when $V_{s'}^{M'}(\psi\wedge\xi) \neq 0$.
			
			\item $\alpha = \psi \rightarrow \xi$: Similar to the previous cases  using induction hypothesis on $\xi$ and $\psi$, the statement holds.
			\item$\alpha = B \psi$: 
					If $V_{s'}^{M'}(B \psi) = 0$, we have $r^{M'}(s', s') = 1$, which contradicts (\ref{eq_Bgeqzero}) since $V_{s'}^{M'}(B B \psi) = 0$. Therefore, this case cannot occur. On the other hand, if $V_{s'}^{M'}(B \psi)\neq 0$, then according to the definition, $V_{s'}^{M'}(B \psi) = \max \{1-r^{M'}(s',s'), V_{s'}^{M'}(\psi)\}$. Note that by the induction hypothesis, we have $V_{s'}^{M'}(\psi) = V_{s'}^{M''}(\psi)$. Thus, we can conclude that $\max \{1-r^{M'}(s',s'), V_{s'}^{M'}(\psi)\} = \max \{1-r^{M''}(s',s'), V_{s'}^{M''}(\psi)\}$, which implies $V_{s'}^{M'}(B \psi) = V_{s'}^{M''}(B \psi)$.
				
		\end{itemize}

			Now, given that $V_{s'}^{M''}(\alpha) = 0$ and $M''$ has more than one state, we can construct the model $M^*$ from $M''$ as described earlier, which leads to a contradiction with (\ref{eq_Bgeqzero}).  Hence, for all models $M$ and all states $s$, we have $0 < V_s^M(\alpha)$.
			\\
			By the induction hypothesis on $\alpha$, we conclude that $\alpha \in \mathcal{T} \cup \mathcal{E} \cup \mathcal{O}$. If $\alpha \in \mathcal{T}$, then $B \alpha \in \mathcal{T}$. We  demonstrate that the other cases cannot occur, thus completing the proof.

		\textsc{Claim}: If $\alpha \in \mathcal{E}\cup \mathcal{O}$, then there is a model $M$ and a state $s$ such that $V_s^M(B\alpha)=0$.
		
		\textit{Proof of claim:}
			By the definitions of $\mathcal{E}$ and $\mathcal{O}$, and referring to Proposition \ref{rem:existence_of_reflexvie_model_for_E_and_O}, we can find a reflexive model $M$ and a state $s$ such that $(M, s) \nvDash \alpha$. Let us consider the model $M = (\mathbb{N}, r, \pi)$ introduced in Lemma \ref{lemma:notinTau_leads_to_have_model}. There exists an $n \in \mathbb{N}$ such that $\max\{1-r(n, m), V_m^M(\alpha)\} \leq \frac{1}{m}$ for infinitely many $m \in \mathbb{N}$. Consequently, we have $V_n^M(B\alpha) = \inf_{m \in \mathbb{N}} \max\{1-r(n, m), V_m^M(\alpha)\} = 0$.
		
	 \hfill$\blacksquare_\textit{Claim}$
	
		Thus by above claim, there is a model $M$ in which $V^{M}_s(B\alpha)=0$, which is a contradiction with the assumption that $\neg B\alpha\equiv \perp$. 		
		
		\vspace{2mm}

		$(\Leftarrow)$ For another side, suppose that $\varphi \in \mathcal{T} \cup \mathcal{E} \cup \mathcal{O}$. We consider the following cases:
		
		case 1. $\varphi \in \mathcal{T}$. By Theorem \ref{soundness}, we know that $\vDash \varphi$, thus, $\forall M\; \forall s\; V_s(\neg \varphi)=0$.
		
		case 2. 
		$\varphi = \mu \rightarrow \xi \in \mathcal{E}$. 
			By Definition \ref{def_4_sets}, we have three possible cases for $\varphi$: $\varphi \approxeq (\neg \neg \psi \wedge \chi)\rightarrow \psi$, $\varphi \equiv \nn \psi \rightarrow \psi$, or $\varphi \approxeq \varphi_e$, where $\psi$ and $\varphi_e$ satisfy the conditions of the definition.
			If $\varphi \approxeq (\neg \neg \psi \wedge \chi)\rightarrow \psi$ or $\varphi \equiv \nn \psi \rightarrow \psi$, it follows that $\neg \varphi \equiv \perp$. Therefore, for all models $M$ and states $s$, we have $V_s(\varphi) \neq 0$, since if $V_s(\psi) = 0$, then $V_s(\nn \psi) = V_s(\nn \psi \wedge \chi) = 0$.
			If $\varphi \approxeq \varphi_e$, the  argument works using recursive nature of definitions.  
	
		case 3. $\varphi \in \mathcal{O}$ and $\varphi = \psi \wedge \chi$, where $\psi , \chi \in \mathcal{T}\cup \mathcal{E}$. By a discussion similar to the cases 1 and 2, it is easy to see that for all models $M$ and all states $s$ we have $V_s(\varphi)>0$. 
	\end{proof}

	\begin{theorem}\label{completeness}\textbf{(Completeness)}
		$\textbf{K}_\textbf{F}$ is complete with respect to the class of all  EGL-models, i.e. for each  formula $\varphi$, if $\vDash \varphi$ then $\vdash \varphi$.
	\end{theorem}
	\begin{proof}
			
			Let $\varphi$ be a formula such that $\vDash \varphi$. By Semi-Completeness (Theorem \ref{weakcompleteness}), we have $\vdash \neg \neg \varphi$.  First, we show that $\vdash \neg \neg \varphi \rightarrow \varphi.$
			%
			Note that from $\vDash \varphi$, we can derive $\vDash \neg\neg \varphi \rightarrow \varphi$. Utilizing Lemma \ref{rem:EuO_arent_valid}, we can deduce that $\neg\neg \varphi \rightarrow \varphi \notin \mathcal{E}\cup \mathcal{O}$; otherwise, there would exist a model $M=(S, r, \pi)$ where, for some $s\in S$, we would have $V_s(\neg\neg \varphi \rightarrow \varphi) < 1$, leading to $V_s(\varphi)<1$, which contradicts $\vDash \varphi$. Additionally, straightforward verification shows that $\neg(\neg\neg \varphi \rightarrow \varphi)\equiv \perp$. Therefore, based on Theorem \ref{NotEqualToPerp}, we can conclude that $\neg \neg \varphi \rightarrow \varphi \in \mathcal{T}$, and so $\neg \neg \varphi \rightarrow \varphi$ is provable.	
			Now, by applying rule \textbf{(R1)} we can derive $\vdash \varphi$.

	\end{proof}

	Although we have established the completeness of $\textbf{K}_\textbf{F}$, the question of whether it is strongly complete remains unanswered.
%
	\begin{corollary}
		$\quad$
		\begin{itemize}[nolistsep]
			\item[(i)] $\textbf{B}_\textbf{F}$ is complete with respect to the class of all  serial and transitive models.
			\item[(ii)] $\textbf{T}_\textbf{F}$ is complete with respect to the class of all  serial, transitive and reflexive models.
		\end{itemize}

	\end{corollary}
	\begin{proof} 
	Since both $\textbf{B}_\textbf{F}$ and $\textbf{T}_\textbf{F}$ are semi-complete with respect to the classes of formulas mentioned, we can establish their completeness using Theorem \ref{NotEqualToPerp} and a similar proof as presented in Theorem \ref{completeness}.
	\end{proof}

		\begin{theorem}(deduction theorem) \label{deduction}
		Let $\Gamma\cup\{\varphi,\psi\}$ be a set of formulas.  $\Gamma\cup \{\varphi\}\vdash \psi $ if and only if $\Gamma \vdash \varphi \rightarrow \psi$.
	\end{theorem}
	\begin{proof}
		First, note that the ``only if" part is trivial. We prove the other direction  by induction on the length of the proof of $\Gamma \cup \{\varphi\}\vdash \psi$. For base step,  if $\psi$ is either an element of $\Gamma$ or is equal to $\varphi$ or an axiom of the corresponding axiomatic system, it can be easily check that  $\Gamma \vdash \varphi \rightarrow \psi$. For induction step, $\psi$ can be deduced by applying either rule R1 or rule R2. In former case $\Gamma \vdash \varphi\rightarrow \psi$ is obtained
		since G\"odel logic preserve deduction theorem. (Theorem 4.2.10 in \cite{Hajek1998}). 
		Now, suppose that the last applied rule is the  necessiation rule. So there is a formula $\chi$ such that $\psi = B \chi$ obtained from rule \textbf{(R2)}.
		But by Definition \ref{def:derivation} we have $\vdash B\chi$ which is followed from $\vdash \chi$.
		Hence,  $\Gamma \vdash B\chi$. Also, $B \chi \rightarrow (\varphi \rightarrow B \chi)$ is provable in G\"odel logic, and we have $\Gamma \vdash B \chi \rightarrow (\varphi \rightarrow B \chi)$. Therefore, by using \textbf{(R1)} we obtain $\Gamma \vdash \varphi \rightarrow B \chi$.
	\end{proof}

	\section{Crisp Models and Finite Model Property}\label{section:finite_model_property}
	
In this section, we observe that the concept of validity in our logics differs from that in crisp models. Furthermore, we demonstrate that our logic does not possess the finite model property. Consequently, we draw a comparison between our logic and the standard G\"odel modal logics $\bm{\mathcal{G}}_\Box$ and $\bm{\mathcal{G}}_\Diamond$ proposed in \cite{CR10}. It is worth noting that $\bm{\mathcal{G}}_\Diamond$ lacks the finite model property, while the validity in $\bm{\mathcal{G}}_\Box$ can be reduced to models with crisp accessibility relations.

	\begin{definition}
A model $M = (S, r, \pi)$ is referred to as \emph{crisp} if its accessibility relation $r$ takes values from the set $\{0,1\}$. i.e. $r: S \times S \rightarrow \{0, 1\}$. We use the notation $\vDash_{\text{crisp}} \varphi$ to indicate that a formula  $\varphi$ is valid within the class of all crisp models.
	\end{definition}

	\begin{example}
		We demonstrate that $\vDash \varphi$ is not equivalent to $\vDash_{crisp} \varphi$. To illustrate this, we consider the formula $\neg\neg B \neg p \rightarrow B \neg p$ and show that it is valid in all crisp models but not in all EGL-models.
		It is not difficult to verify that in all crisp models $M=(S,r,\pi)$, for all $s\in S$, we have $V_s^M(\neg\neg B \neg p) = V_s^M(B\neg p)$, and consequently, $V_s^M(\neg\neg B \neg p \rightarrow B \neg p)=1$. This implies $\vDash_{crisp} \neg\neg B \neg p \rightarrow B \neg p$.
		However, let us consider the  single-state model $M'=({s_1}, r', \pi')$ where $r(s_1, s_1)=0.2$ and $\pi(s_1,p)=0.2$. By a simple check, we find that $V_{s_1}^{M'}(\neg\neg B \neg p \rightarrow B \neg p) = 0.8$, indicating that $\nvDash \neg\neg B \neg p \rightarrow B \neg p$. 
			\end{example}

Now we show that our logic does not have the finite model property.


The following Proposition shows that it is a valid formula in the class of finite models.

\begin{proposition}\label{axim6isvalid}
	The formula $\neg \neg B\neg \neg \varphi\rightarrow \neg \neg B \varphi$ is true in any model $M=(S, r, \pi)$ in which $S$ is finite.
\end{proposition}
\begin{proof}
	Indeed, it can be easily checked that the schema $\neg \neg B\neg \neg \varphi\rightarrow \neg \neg B \varphi$,  is not true in the infinite model $M=(\mathbb{N}, r, \pi)$, where $\mathbb{N}$ represents the set of natural numbers, $r(i,j)=1$ for all $i,j\in\mathbb{N}$, and $V_i(\varphi)=\frac{1}{i+1}$ for all $i\in\mathbb{N}$ (as introduced in Section 5 of \cite{CR10}).
	Now, assume $S$ is finite and $s\in S$ be a state. We can easily see that since $S$ is finite then $V_s(B\neg\neg\varphi)=0$ if and only if $V_s(B\varphi)=0$. Moreover, $V_s(B\neg\neg\varphi)>0$ if and only if (for all $s'\in S$, $V_{s'}(\neg\neg\varphi)=1$ or $r(s,s')< 1$) if and only if  (for all $s'\in S$, $V_{s'}(\varphi)>0$ or $r(s,s')< 1$) if and only if $V_s(B\varphi)>0$. Thus, in both cases, we can  conclude that $V_s(\neg\neg B\neg\neg\varphi)=V_s(\neg\neg B\varphi).$
\end{proof}

\begin{remark}
	  Considering a model with a single state $s$ in which $ \pi(s,p)=0 $ for some $ p\in\mathcal{P} $ and $ r_a(s,s)=0.2 $ for an agent $ a\in\mathcal{A} $, we can see that $Z_\Box: \neg \neg \Box \varphi \rightarrow \Box \neg \neg \varphi$ (the axiom which is used in \cite{CR10}) is not valid even in the class of finite models.
\end{remark}
	
	\section{Conclusion}

We  introduced an epistemic extension $\textbf{K}_\textbf{F}$ of G\"odel fuzzy logic that serves as a fuzzy version of classical epistemic logic $\textbf{K}$. Also, we proposed two axiomatic systems  $\textbf{B}_\textbf{F}$ and $\textbf{T}_\textbf{F}$ by considering consistent belief and adding positive introspection and Truth axioms to the axioms of $\textbf{K}_\textbf{F}$, respectively.
We proved that all of these systems are sound and complete with respect to the appropriate Kripke-based fuzzy models.  	 Furthermore, we have revealed that validity in $\textbf{K}_\textbf{F}$ cannot be reduced to the class of all models having crisp accessibility relations, and also  $\textbf{K}_\textbf{F}$ does not enjoy the finite model property. Then we  concluded that these properties distinguish $\textbf{K}_\textbf{F}$ as a new epistemic extension of   G\"odel fuzzy logic that differs from the standard G\"odel Modal Logics $\bm{\mathcal{G}}_\Box$ and $\bm{\mathcal{G}}_\Diamond$ proposed by Caicedo and O. Rodriguez in \cite{CR10}.

\section*{Acknowledgement}
We would like to express my sincere gratitude to Professor Hans van Ditmarsch for his invaluable and insightful comments on the manuscript. His thoughtful feedback and constructive suggestions have significantly enriched the quality of this work. We are truly appreciative of the time and effort he dedicated to providing detailed and constructive feedback, which has undoubtedly enhanced the clarity and depth of the content.

	\section*{Appendix}
	\textbf{Proof of Lemma \ref{lemma:notinTau_leads_to_have_model}}.
	\begin{proof}
		We use induction on the complexity of $\varphi$. Notice that for the following case 1 and case 2, we only need the weaker assumption, i. e. $\nvDash \varphi$.\\
		\textbf{case 1.} $\varphi = p \;(p\in\mathcal{P})$.  Let $M=(\mathbb{N}, r, \pi)$, in which $r(n,m)=1$ and $\pi(m,p)\leq \frac{1}{p}$ for all $n,m\in\mathbb{N}$.\\
		\textbf{case 2.} 
		$\varphi = \neg \psi$. Since $\nvDash \varphi$, we have $\nvDash \neg \psi$, which implies $\psi\not\equiv \perp$. Therefore, there exists a model $M$ and a state $s\in S^{M}$ such that $V_{s}^M(\psi)>0$.
		Now, for each $i\in \mathbb{N}$, let $M_i$ be a copy of $M$. We construct the model $M'$ as follows:
		
		$$\begin{array}{l}
			S^{M'} = \bigcup_{j\in \mathbb{N}}S^{M_i}\cup \{t\}, \\
			r^{M'}(t, u) = 1-\frac{1}{i}, \quad r(u, t)=0 \quad \forall i\in \mathbb{N}, u\in S^{M_i}.\\
			r^{M'}(u,v) = r^{M_i}(u,v),\quad \forall u,v\in S^{M_i},\\
			r^{M'}(u, v) = 0 \quad \forall u \in S^{M_i} \forall v \in S^{M_j},\\
			r^{M'}(t,t) = 1.
		\end{array}$$
		
		It is not difficult to verify that for each $i\in \mathbb{N}$, we have $\max\{1-r^{M'}(t,s_{i}), V_{s_i}^{M'}(\neg \psi)\}\leq \frac{1}{i}$, where $s_i$ is the copy of $s$ in $M_i$ (See Figure \ref{fig:case3-lemma}).  Indeed, we only check the case $\psi = B\chi$, and we can rearrange the names of the states of $M'$ to obtain the desired model.
		In the case of $\psi = B\chi$, we can observe that
		
		$$\begin{array}{ll}
			V_{s_i}^{M'} (\psi) & = \inf_{u\in S^{M'}} \max \{1-r^{M'}(s_i, u), V_u^{M'}(\chi)\},\\
			& = \inf\left\lbrace  \inf_{u\in S^{M_j}}\max\{1-r^{M'}(s_i, u), \right. V_u^{M'}(\chi)\} ,\\ &\quad \max\{1-r^{M'}(s_i, t), V_t^{M'}(\chi)\}, \\
			& \quad  \inf_{u\in S^{M'}\backslash (\bigcup_j S^{M_j}\cup t)}\left.\max \{1-r^{M'}(s_i,u), V_u^{M'}(\chi)\}  \right\rbrace\\
			& = \inf_{u\in S^{M_i}}\max\{1-r^{M_i}(s_i, u), V_u^{M_i}(\chi)\} \\
			& = \inf_{u\in S^{M}}\max\{1-r^{M}(s, u), V_u^{M}(\chi)\}\\
			& = V_s^M(B\chi) \\
			& = V_s^M(\psi).
		\end{array}$$
		
		Now, considering the fact that $V_s^M(\psi)>0$, it follows that $V_{s_i}^{M'}(\neg \psi) = 0$ for all $i \in \mathbb{N}$. Since $r^{M'}(t,s_i)=1-\frac{1}{i}$, we can observe that $\max\{1-r^{M'}(t,s_i), V_{s_i}^{M'}(\neg \psi)\}\leq \frac{1}{i}$.
		\\
		\tikzset{every picture/.style={line width=0.75pt}} 
		\begin{figure}
			\begin{center}
				\resizebox*{10cm}{!}{
					\begin{tikzpicture}[x=0.75pt,y=0.75pt,yscale=-1,xscale=1]
						
						\draw  [fill={rgb, 255:red, 255; green, 247; blue, 195 }  ,fill opacity=1 ] (127.7,132.15) .. controls (127.7,125.99) and (132.69,121) .. (138.85,121) .. controls (145.01,121) and (150,125.99) .. (150,132.15) .. controls (150,138.31) and (145.01,143.3) .. (138.85,143.3) .. controls (132.69,143.3) and (127.7,138.31) .. (127.7,132.15) -- cycle ;
						\draw   (66.2,106.3) -- (168.2,106.3) -- (168.2,164.3) -- (66.2,164.3) -- cycle ;
						\draw  [fill={rgb, 255:red, 255; green, 247; blue, 195 }  ,fill opacity=1 ] (284.7,243.15) .. controls (284.7,236.99) and (289.69,232) .. (295.85,232) .. controls (302.01,232) and (307,236.99) .. (307,243.15) .. controls (307,249.31) and (302.01,254.3) .. (295.85,254.3) .. controls (289.69,254.3) and (284.7,249.31) .. (284.7,243.15) -- cycle ;
						\draw   (223.2,217.3) -- (325.2,217.3) -- (325.2,275.3) -- (223.2,275.3) -- cycle ;
						\draw  [fill={rgb, 255:red, 255; green, 247; blue, 195 }  ,fill opacity=1 ] (283.7,70.15) .. controls (283.7,63.99) and (288.69,59) .. (294.85,59) .. controls (301.01,59) and (306,63.99) .. (306,70.15) .. controls (306,76.31) and (301.01,81.3) .. (294.85,81.3) .. controls (288.69,81.3) and (283.7,76.31) .. (283.7,70.15) -- cycle ;
						\draw   (222.2,44.3) -- (324.2,44.3) -- (324.2,102.3) -- (222.2,102.3) -- cycle ;
						\draw  [fill={rgb, 255:red, 255; green, 247; blue, 195 }  ,fill opacity=1 ] (283.7,142.15) .. controls (283.7,135.99) and (288.69,131) .. (294.85,131) .. controls (301.01,131) and (306,135.99) .. (306,142.15) .. controls (306,148.31) and (301.01,153.3) .. (294.85,153.3) .. controls (288.69,153.3) and (283.7,148.31) .. (283.7,142.15) -- cycle ;
						\draw   (222.2,116.3) -- (324.2,116.3) -- (324.2,174.3) -- (222.2,174.3) -- cycle ;
						\draw  [fill={rgb, 255:red, 255; green, 247; blue, 195 }  ,fill opacity=1 ] (404.7,172.15) .. controls (404.7,165.99) and (409.69,161) .. (415.85,161) .. controls (422.01,161) and (427,165.99) .. (427,172.15) .. controls (427,178.31) and (422.01,183.3) .. (415.85,183.3) .. controls (409.69,183.3) and (404.7,178.31) .. (404.7,172.15) -- cycle ;
						\draw    (306,70.15) .. controls (345.6,83.1) and (390.63,112.69) .. (414.76,158.88) ;
						\draw [shift={(415.85,161)}, rotate = 243.29] [fill={rgb, 255:red, 0; green, 0; blue, 0 }  ][line width=0.08]  [draw opacity=0] (10.72,-5.15) -- (0,0) -- (10.72,5.15) -- (7.12,0) -- cycle    ;
						\draw    (415.85,161) .. controls (378.76,148.49) and (331.44,125.31) .. (307.1,72.58) ;
						\draw [shift={(306,70.15)}, rotate = 66.21] [fill={rgb, 255:red, 0; green, 0; blue, 0 }  ][line width=0.08]  [draw opacity=0] (10.72,-5.15) -- (0,0) -- (10.72,5.15) -- (7.12,0) -- cycle    ;
						\draw    (306,142.15) .. controls (347.15,135.47) and (377.64,147.66) .. (402.78,170.38) ;
						\draw [shift={(404.7,172.15)}, rotate = 223.09] [fill={rgb, 255:red, 0; green, 0; blue, 0 }  ][line width=0.08]  [draw opacity=0] (10.72,-5.15) -- (0,0) -- (10.72,5.15) -- (7.12,0) -- cycle    ;
						\draw    (404.7,172.15) .. controls (379.13,182.91) and (324.22,163.73) .. (307.65,144.27) ;
						\draw [shift={(306,142.15)}, rotate = 54.83] [fill={rgb, 255:red, 0; green, 0; blue, 0 }  ][line width=0.08]  [draw opacity=0] (10.72,-5.15) -- (0,0) -- (10.72,5.15) -- (7.12,0) -- cycle    ;
						\draw    (307,243.15) .. controls (341.5,196.26) and (367.16,189.56) .. (413.02,183.66) ;
						\draw [shift={(415.85,183.3)}, rotate = 172.82] [fill={rgb, 255:red, 0; green, 0; blue, 0 }  ][line width=0.08]  [draw opacity=0] (10.72,-5.15) -- (0,0) -- (10.72,5.15) -- (7.12,0) -- cycle    ;
						\draw    (415.85,183.3) .. controls (400.36,203.1) and (375.7,234.66) .. (309.03,242.91) ;
						\draw [shift={(307,243.15)}, rotate = 353.43] [fill={rgb, 255:red, 0; green, 0; blue, 0 }  ][line width=0.08]  [draw opacity=0] (10.72,-5.15) -- (0,0) -- (10.72,5.15) -- (7.12,0) -- cycle    ;
						\draw    (425,166.6) .. controls (442.55,139.3) and (460.1,182.35) .. (429.45,172.98) ;
						\draw [shift={(427,172.15)}, rotate = 20.11] [fill={rgb, 255:red, 0; green, 0; blue, 0 }  ][line width=0.08]  [draw opacity=0] (10.72,-5.15) -- (0,0) -- (10.72,5.15) -- (7.12,0) -- cycle    ;
						\draw    (294.85,81.3) .. controls (288.24,96.06) and (282.42,104.03) .. (293.57,128.29) ;
						\draw [shift={(294.85,131)}, rotate = 244.05] [fill={rgb, 255:red, 0; green, 0; blue, 0 }  ][line width=0.08]  [draw opacity=0] (10.72,-5.15) -- (0,0) -- (10.72,5.15) -- (7.12,0) -- cycle    ;
						\draw    (294.85,131) .. controls (310.71,108.43) and (300.78,95.61) .. (295.92,84.08) ;
						\draw [shift={(294.85,81.3)}, rotate = 71.36] [fill={rgb, 255:red, 0; green, 0; blue, 0 }  ][line width=0.08]  [draw opacity=0] (10.72,-5.15) -- (0,0) -- (10.72,5.15) -- (7.12,0) -- cycle    ;
						\draw    (294.85,153.3) .. controls (280.45,173.96) and (282.84,202.81) .. (294.72,229.53) ;
						\draw [shift={(295.85,232)}, rotate = 244.87] [fill={rgb, 255:red, 0; green, 0; blue, 0 }  ][line width=0.08]  [draw opacity=0] (10.72,-5.15) -- (0,0) -- (10.72,5.15) -- (7.12,0) -- cycle    ;
						\draw    (295.85,232) .. controls (312.91,192.14) and (301.56,171.2) .. (295.86,156.12) ;
						\draw [shift={(294.85,153.3)}, rotate = 71.4] [fill={rgb, 255:red, 0; green, 0; blue, 0 }  ][line width=0.08]  [draw opacity=0] (10.72,-5.15) -- (0,0) -- (10.72,5.15) -- (7.12,0) -- cycle    ;
						
						\draw (134,128) node [anchor=north west][inner sep=0.75pt]   [align=left] {s};
						\draw (291,239) node [anchor=north west][inner sep=0.75pt]   [align=left] {s};
						\draw (290,66) node [anchor=north west][inner sep=0.75pt]   [align=left] {s};
						\draw (290,138) node [anchor=north west][inner sep=0.75pt]   [align=left] {s};
						\draw (411.85,167) node [anchor=north west][inner sep=0.75pt]   [align=left] {t};
						\draw (91,128) node [anchor=north west][inner sep=0.75pt]   [align=left] {$\displaystyle M$};
						\draw (229,50) node [anchor=north west][inner sep=0.75pt]   [align=left] {$\displaystyle M_{1}$};
						\draw (228,121) node [anchor=north west][inner sep=0.75pt]   [align=left] {$\displaystyle M_{2}$};
						\draw (229,224) node [anchor=north west][inner sep=0.75pt]   [align=left] {$\displaystyle M_{i}$};
						\draw (331.18,162.22) node [anchor=north west][inner sep=0.75pt]  [font=\tiny,rotate=-357.73,xslant=0.04] [align=left] {$\displaystyle  \begin{array}{{>{\displaystyle}l}}
								1-\frac{1}{2}\\
							\end{array}$};
						\draw (324,8) node [anchor=north west][inner sep=0.75pt]   [align=left] {$\displaystyle M'$};
						\draw (370.18,223.22) node [anchor=north west][inner sep=0.75pt]  [font=\tiny,rotate=-357.73,xslant=0.04] [align=left] {$\displaystyle  \begin{array}{{>{\displaystyle}l}}
								1-\frac{1}{i}\\
							\end{array}$};
						\draw (349,108) node [anchor=north west][inner sep=0.75pt]   [align=left] {\begin{minipage}[lt]{8.67pt}\setlength\topsep{0pt}
								\begin{flushright}
									{\scriptsize 0}
								\end{flushright}
								
						\end{minipage}};
						\draw (355,145) node [anchor=north west][inner sep=0.75pt]   [align=left] {\begin{minipage}[lt]{8.67pt}\setlength\topsep{0pt}
								\begin{flushright}
									{\scriptsize 0}
								\end{flushright}
								
						\end{minipage}};
						\draw (355,195) node [anchor=north west][inner sep=0.75pt]   [align=left] {\begin{minipage}[lt]{8.67pt}\setlength\topsep{0pt}
								\begin{flushright}
									{\scriptsize 0}
								\end{flushright}
								
						\end{minipage}};
						\draw (370,88) node [anchor=north west][inner sep=0.75pt]   [align=left] {\begin{minipage}[lt]{8.67pt}\setlength\topsep{0pt}
								\begin{flushright}
									{\scriptsize 0}
								\end{flushright}
								
						\end{minipage}};
						\draw (250,180) node [anchor=north west][inner sep=0.75pt]   [align=left] {$\vdots$};
						\draw (260,285) node [anchor=north west][inner sep=0.75pt]   [align=left] {$\vdots$};
						\draw (300,103) node [anchor=north west][inner sep=0.75pt]   [align=left] {\begin{minipage}[lt]{8.67pt}\setlength\topsep{0pt}
								\begin{flushright}
									{\scriptsize 0}
								\end{flushright}
								
						\end{minipage}};
						\draw (272,103) node [anchor=north west][inner sep=0.75pt]   [align=left] {\begin{minipage}[lt]{8.67pt}\setlength\topsep{0pt}
								\begin{flushright}
									{\scriptsize 0}
								\end{flushright}
								
						\end{minipage}};
						\draw (300,187) node [anchor=north west][inner sep=0.75pt]   [align=left] {\begin{minipage}[lt]{8.67pt}\setlength\topsep{0pt}
								\begin{flushright}
									{\scriptsize 0}
								\end{flushright}
								
						\end{minipage}};
						\draw (271,188) node [anchor=north west][inner sep=0.75pt]   [align=left] {\begin{minipage}[lt]{8.67pt}\setlength\topsep{0pt}
								\begin{flushright}
									{\scriptsize 0}
								\end{flushright}
								
						\end{minipage}};
						\draw (443,157) node [anchor=north west][inner sep=0.75pt]   [align=left] {\begin{minipage}[lt]{8.67pt}\setlength\topsep{0pt}
								\begin{flushright}
									{\scriptsize 1}
								\end{flushright}
								
						\end{minipage}};

					\end{tikzpicture}
				}
			\end{center}
			\caption{The left part is the structure of model $M$, and the right part is the structure of the model $M'$ described in case 2 of Lemma  \ref{lemma:notinTau_leads_to_have_model} in which each $M_i$ is a copy of $M$.} \label{fig:case3-lemma}
		\end{figure}
		\textbf{case 3.} $\varphi = \psi \wedge \chi$.
		Suppose that for a reflexive model $M$ and  a state $s\in S^M$, we have $(M,s)\nvDash \varphi$. Consequently, $(M,s)\nvDash \psi$ or $(M,s)\nvDash \chi$.
		Let's assume, without loss of generality,  $(M,s)\nvDash\psi$.
		By induction hypothesis there is a reflexive model $M'=(\mathbb{N}, r, \pi)$ and an 
		$n\in\mathbb{N}$ such that for infinitely many $m\in\mathbb{N}$ $\max\{1-r^{M'}(n,m), V_m^{M'}(\psi)\}\leq\frac{1}{m}$. It can be easily verified that in this model, for infinitely many $m\in \mathbb{N}$, $\max\{1-r^{M'}(n,m), V_m^{M'}(\psi\wedge\chi)\}\leq\frac{1}{m}$.\\
		\textbf{case 4.}
		$\varphi=\psi\rightarrow\chi$. 
		From $\nvDash \varphi$, we can infer $\nvDash \chi$. Let's define $A_n^M$, $B_n^M$, and $E$ as follows:
		$$A_n^M = \{m\in \mathbb{N}\mid \max\{1-r^M(n,m), V^M_m(\chi)\}\leq \frac{1}{m}\},$$
		$$B_n^M = \{m\in A_n^M\mid V^M_m(\psi)>V^M_m(\chi)\}.$$
		$$E = \{(M,n)\,\mid\, M =(\mathbb{N}, r, \pi)\, \text{ is a reflexive model}, n\in \mathbb{N},  A_n^M \text{ has infinite elements}\},$$ 
		By the induction hypothesis, we can conclude that the set $E$ is not empty. Our \textbf{claim} is that there exists $(M,n)\in E$ such that $B_n^M$ is infinite. It is important to note that this particular $(M,n)$ will satisfy $\max\{1-r^M(n,m), V_m^M(\psi\rightarrow\chi)\}\leq \frac{1}{m}$ for infinitely many $m\in \mathbb{N}$, since $V_m^M(\psi\rightarrow\chi)=V_m^M(\chi)$ for $m\in B_n^M$.
		
		\vskip 0.2cm
		\textbf{proof of the claim.} 
		
		\textbf{step 1.}
		There is $(M,n)\in E$ such that $B_n^M \neq \emptyset$.\\
		Assume that for all $(M,n)\in E$, $B_n^M = \emptyset$. So, for all $(M,n)\in E$ and $m\in \mathbb{N}$,  we have $m\notin A_n^M \text{ or } V_m^M(\psi)\leq V_m^M(\chi)$. In other words, for all $(M,n)\in E$ and $m\in \mathbb{N}$
		
		\begin{align}
			m\in A_n^M \Rightarrow V_m^M(\psi)\leq V_m^M(\chi). \label{eq_028}
		\end{align}
		Moreover, we can show that for all $(M,n)\in E$ and $m\in\mathbb{N}$ we have
		\begin{align}
			m \notin A_n^M \Rightarrow V_m^M(\psi)\leq V_m^M(\chi). \label{eq_029}
		\end{align}
		
		Indeed, suppose that there exists $(M^*, n^*)\in E$ and $m^*\in\mathbb{N}$ such that $m^*\notin A_{n^*}^{M^*}$ and $V_{m^*}^{M^*}(\psi) > V_{m^*}^{M^*}(\chi)$. This implies that $\max\{1-r^{M^*}(n^*, m^*), V_{m^*}^{M^*}(\chi)\} > \frac{1}{m^*}$ and $V_{m^*}^{M^*}(\psi) > V_{m^*}^{M^*}(\chi)$.
		Now, let's consider $z\in \mathbb{N}$ such that $\max\{1-r^{M^*}(n^*, m^*), V_{m^*}^{M^*}(\chi)\}\leq \frac{1}{z}$ (we can choose $z=1$ in this case). We can create a reflexive model $M_1$ from $M^*$ by relabeling the states such that the state originally labeled $m^*$ is now labeled $z$. Consequently, $\max\{1-r^{M_1}(n^*, z), V_z^{M_1}(\chi)\}\leq \frac{1}{z}$ and $V_z^{M_1}(\psi) > V_z^{M_1}(\chi)$. However, this implies that $z\in B_{n^*}^{M_1}$, which contradicts our initial assumption. Hence, Equation (\ref{eq_029}) holds.
		From Equation (\ref{eq_028}) and Equation (\ref{eq_029}), we can conclude that for any reflexive model $M=(\mathbb{N}, r, \pi)$ and any $m\in \mathbb{N}$, $V_m^M(\psi)\leq V_m^M(\chi)$. However, this contradicts our assumption that there exists a reflexive model $M$ and a state $s\in S^M$ where $(M,s)\nvDash \varphi$.
		
		\textbf{step 2.}		
		There is $(M,n)\in E$ such that $\mid B_n^M\mid\geq 2$.\\
		By step 1, there is $(M_1, n_1)\in E$, and $m_1\in B_{n_1}^{M_1}$. Let as a copy of $(M_1,n_1)$, we have $(M_2, n_2)\in E$ and $m_2\in B_{n_2}^{M_2}$. We build a model $M$ as the following:
		\begin{eqnarray*}
			& &S^{M} = S^{M_1}\cup S^{M_2,}\cup \{t\},\\
			& &\forall u\in S^{M_1}, \forall v\in S^{M_2}\;\;\;\; r^{M}(u, v) = r^{M}(v, u) = 0,\\
			& &\forall u, v \in  S^{M_1}\;\;\;\; r^{M}(u,v) = r^{M_1}(u,v),\\
			& &  \forall u, v \in  S^{M_2}\;\;\;\; r^{M}(u,v) = r^{M_2}(u,v),\\
			& &r^{M} (t, m_1) = 1-V_{m_1}(\chi),  \quad  r^{M} (t, m_2) = 1-V_{m_2}(\chi),\\
			& &r^{M} (m_1, 1) =  r^{M} (m_2, t) = 0,\\
			& &\forall u \in S^{M}\backslash\{m_1,m_2\}\;\;\;\; r^{M}(u,t)=r^{M}(t,u) = 0.
		\end{eqnarray*}
		
		The construction of the model $M$ ensures that all previous values of formulas in both model $M_1$ and $M_2$ are preserved. Furthermore, by renaming the states such that $S^{M}=\mathbb{N}$, we obtain new natural numbers $m_1', m_2'$ that correspond to the original $m_1$ and $m_2$, respectively and both $m_1'$ and $m_2'$ belong to $A_t^M$ and $B_t^M$.
		
		\textbf{step 3.}
		Constructing a model $M'$ from $M$ such that for infinitely many $p\in S^{M'}$
		$V_p^{M'}(\psi) > V_p^{M'}(\chi)$.\\
		Following step 2, let $(M, n) \in E$, and choose $m_0 \in B_n^M$ such that $m_0 \geq 2$. Thus, we have a reflexive model $M = (\mathbb{N}, r, \pi)$ where $m_0, n \in S^M$ (with $m_0 \geq 2$), satisfying $\max\{1-r^M(n,m_0), V_{m_0}^M(\chi)\}\leq \frac{1}{m_0}$ and $V_{m_0}^N(\psi)> V_{m_0}^M(\chi)$.
		We extend $M$ to $M'$ using the following procedure:
		
		$$\begin{array}{ll}
			\multicolumn{2}{l}{ S^{M'} = S^{M} \cup \bigcup_{j\in \mathbb{N}, j\geq2} \{m_0^j\} } \\
			\forall m,m' \in \mathbb{N} & r^{M'}(m,m') = r^M(m,m') ,\\
			\forall  j\geq 2,m\in \mathbb{N} & r^{M'}(m,m_0^j) = 0,\\
			& r^{M'}(m_0^j, m) = r^{M}(m_0, m) \\
			\forall j,j'\in \mathbb{N}& j\neq j' \quad r^{M'}(m_0^j,m_0^{j'})=0, \\ 
			& r^{M'}(m_0^j, m_0^j)=1,\\
			& r^{M'}(n,m_0^j) = r^M(n,m_0) - \frac{1}{j}.\\
			\forall j\in \mathbb{N}, & \pi^{M'}(j,p) = \pi^{M}(j,p), \\
			\forall j\geq 2 &  \pi^{M'}(m_0^j, p) = \pi^M(m_0,p)\\ 
		\end{array}$$
		In this new model $M'$, we introduce copies of $m_0$ denoted as $m_0^j$, and note that $M$ is a submodel of $M'$. We will now demonstrate that the following property holds for all formulas $\chi'$:
		
		\begin{equation}\label{eq019}
			\forall m\in \mathbb{N} \quad V_m^{M'}(\chi') = V_m^{M}(\chi') \quad \text{ and }\quad \forall j\geq 2 \quad V_{m_0^j}^{M'} (\chi') = V_{m_0}^M(\chi').
		\end{equation}
		
		We use induction on the complexity of $\chi'$. The base case, $\chi' = p$, is straightforward to prove using the definition of $M'$. Since the cases $\chi' = \neg \xi$, $\chi' = \xi \wedge \gamma$, and $\chi' = \xi \rightarrow \gamma$ do not depend on the accessibility relation, we only need to prove the case $\chi' = B \xi$. We will further divide this case into three following subcases to establish our proof:
		
		\begin{itemize}
			\item If $m\in \mathbb{N}$ and $m\neq n$. We have :
			\begin{align}
				&V_m^{M'}(B\xi) && = \inf_{m'\in S^{M'}} \max\{1-r^{M'}(m,m'), V_{m'}^{M'}(\xi)\} \nonumber\\
				&&&= \inf_{m'\in S^{M}} \max\{1-r^M(m,m'), V_{m'}^{M'}(\xi)\} \label{eq_007}\\
				&&&=\inf_{m'\in S^{M}} \max\{1-r^M(m,m'), V_{m'}^{M}(\xi)\} \label{eq_008} \\
				&&&=V_m^M(B\xi). \nonumber
			\end{align}
			The line (\ref{eq_007}) obtained from the definition of accessibility relations in $M'$ and (\ref{eq_008}) obtained by induction hypothesis.
			
			\item If $j\geq 2, j\in \mathbb{N}$. Consider the following sets:
			\begin{align}
				& C_1 = \bigcup_{k\in \mathbb{N}} \max\{1-r^{M'}(m_0^j,k), V_k^{M'}(\xi)\}, & \label{eq_0020} \\
				& D_1 = \bigcup_{k\in \mathbb{N}, k\geq2} \max\{1-r^{M'}(m_0^j, m_0^k), V_{m_0^k}^{M'}(\xi)\} \label{eq_020} \\
				& C_2 = \bigcup_{k\in \mathbb{N}} \max\{1-r^{M}(m_0,k), V_k^{M}(\xi)\}, &
				\label{eq_0021} \\
				& D_2 = \bigcup_{k\in \mathbb{N}, k\geq2} \max\{1-r^{M'}(m_0^j, m_0^k), V_{m_0^k}^{M}(\xi)\}. \label{eq_021}
			\end{align}
			We have:
			\begin{align}
				&V_{m_0^j}^{M'}(B\xi) && = \inf_{m'\in S^{M'}} \max\{1-r^{M'}(m_0^j,m'), V_{m'}^{M'}(\xi)\} \nonumber\\
				&&&= \inf  C_1 \cup D_1 \label{eq_009}\\
				&&&=\inf  C_2 \cup D_2 \label{eq_010} \\
				&&&=\inf  C_2 \label{eq_011} \\
				&&&=V_{m_0}^M(B\xi). \nonumber
			\end{align}
			Line (\ref{eq_009}) is obtained by using equations (\ref{eq_0020}) and (\ref{eq_020}). Similarly, line (\ref{eq_010}) is derived from the induction hypothesis and the definition of accessibility relations in $M'$, and equations (\ref{eq_0021}) and (\ref{eq_021}). Moving on to line (\ref{eq_011}), it follows from the fact that for all $k \in \mathbb{N}$ with $k \geq 2$, we have $\max\{1-r^{M'}(m_0^j, m_0^k), V_{m_0^k}^{M}(\xi)\}=1$. As a result, these terms do not have any impact on the infimum.
			
			\item For the remaining state $n\in \mathbb{N}$, similar to the previous subcase, we define the following sets:
			\begin{align}
				& C_3 = \bigcup_{k\in \mathbb{N}} \max\{1-r^{M'}(n,k), V_k^{M'}(\xi)\}, &
				\label{eq_0022}\\
				& D_3 = \bigcup_{k\in \mathbb{N}, k\geq2} \max\{1-r^{M'}(n, m_0^k), V_{m_0^k}^{M'}(\xi)\} \label{eq_022} \\
				& C_4 = \bigcup_{k\in \mathbb{N}} \max\{1-r^{M}(n,k), V_k^{M}(\xi)\}, &  \label{eq_0023}\\
				& D_4 = \bigcup_{k\in \mathbb{N}, k\geq2} \max\{1-r^{M}(n, m_0)+\frac{1}{k}, V_{m_0}^{M}(\xi)\}. \label{eq_023}
			\end{align}
			We have:
			\begin{align}
				&V_{n}^{M'}(B\xi) && = \inf_{m'\in S^{M'}} \max\{1-r^{M'}(n,m'), V_{m'}^{M'}(\xi)\} \nonumber\\
				&&&= \inf  C_3 \cup D_3 \label{eq_012}\\
				&&&=\inf  C_4 \cup D_4 \label{eq_013} \\
				&&&=\inf  C_4 \label{eq_014} \\
				&&&=V_{n}^M(B\xi). \nonumber
			\end{align}	
			Line (\ref{eq_012}) is derived using equations (\ref{eq_0022}) and (\ref{eq_022}). Similarly, line (\ref{eq_013}) is obtained from the induction hypothesis, the definition of accessibility relations in $M'$, and equations (\ref{eq_0023}) and (\ref{eq_023}). Moving on to line (\ref{eq_014}), it follows from (\ref{eq_023}) and the fact that for all $k \in \mathbb{N}$ with $k \geq 2$, we have:
			$$\max\{1-r(n,m_0), V_{m_0}(\xi)\} \leq  \max\{1-r(n,m_0)+\frac{1}{k}, V_{m_0}(\xi)\}. $$
		\end{itemize}
		
		\textbf{step 4}. Constructing a model $M''$ from $M'$ such that $B_n^{M''}$ has infinitely many numbers.\\
		For all states $m_0^k$ in $M'$ $(k\geq 2, k\in \mathbb{N})$ there is an  $\alpha\in \mathbb{N}$ such that:
		$$\max\{1-r^{M'}(n,m_0^k), V_{m_0^k}^{M'}(\chi)\}\leq \alpha,$$
		To fix some $t$ $(t\geq 1, t\in\mathbb{N})$ such that $\alpha\leq \frac{1}{t}$ we need to choose some suitable $\beta$ and $\beta'$ as follows:
		$$\alpha \leq \left \lbrace \begin{array}{ll}
			V_{m_0}^{M'}(\chi) + \frac{1}{k+\beta} & V_{m_0^k}^{M'}(\chi) > 1-r^{M'}(n,m_0^k) \\
			1-r^M(n,m_0)+\frac{1}{k+\beta'} & V_{m_0^k}^{M'}(\chi) \leq 1-r^{M'}(n,m_0^k),
		\end{array}\right.$$
		equivalently
		$$\alpha \leq \left \lbrace \begin{array}{ll}
			\frac{1}{\frac{k+\beta}{(k+\beta) V_{m_0}(\chi)+1}} & V_{m_0^k}^{M'}(\chi) > 1-r^{M'}(n,m_0^k) \\
			\frac{1}{\frac{k+\beta'}{(k+\beta')(1-r^M(n,m_0))+1}} & V_{m_0^k}^{M'}(\chi) \leq 1-r^{M'}(n,m_0^k).
		\end{array}\right.$$	
		So for the first case we find $\beta$ for which $t = [{\frac{(k+\beta)}{(k+\beta) V_{m_0}(\chi)+1}}]> 1$,
		and for the second case we find $\beta'$ for which 
		$t= [{\frac{(k+\beta')}{(k+\beta')(1-r^M(n,m_0))+1}}]> 1$.\\
		For the first case we have	
		\begin{align*}
			& {\frac{(k+\beta)}{(k+\beta) V_{m_0}(\chi)+1}} > 1  & \iff \\
			& (k+\beta) > (k+\beta) V_{m_0}(\chi)+1  & \iff\\
			& (k+\beta)(1-V_{m_0}(\chi)) > 1  & \iff \\
			& (k+\beta) > \frac{1}{1-V_{m_0}(\chi)} & \iff \\
			& \beta > \frac{1}{1-V_{m_0}(\chi)} - k & \iff \\
			& \beta > \frac{1-k-kV_{m_0}(\alpha)}{1-V_{m_0}(\chi)}.
		\end{align*}
		
		Note that the denominator is well defined because of the assumption that $\chi$ is not satisfiable in the state $m_0$. Therefore, if we choose $\beta > \frac{1-k-kV_{m_0}(\alpha)}{1-V_{m_0}(\chi)}$, then $t = \left[\frac{(k+\beta)}{(k+\beta) V_{m_0}(\chi)+1}\right] > 1$, satisfying the desired condition. Similarly, for the second case, we proceed as follows:
		
		\begin{align*}
			& {\frac{(k+\beta')}{(k+\beta')(1-r^M(n,m_0))+1}} > 1  & \iff \\
			& (k+\beta') > (k+\beta')(1-r^M(n,m_0))+1  & \iff\\
			& (k+\beta')(1-(1-r^M(n,m_0))) > 1  & \iff \\
			& (k+\beta') > \frac{1}{r^M(n,m_0)} & \iff \\
			& \beta' > \frac{1}{r^M(n,m_0)} - k & \iff \\
			& \beta' > \frac{1-k\,r^M(n,m_0)}{r^M(n,m_0)}.
		\end{align*}
		
		Note that $r^M(n,m_0)\neq 0$ since $m_0\geq 2$ and $\max\{1-r^M(n,m_0), V_{m_0}^M(\chi)\}\leq \frac{1}{m_0}$. Now, to construct the model $M''$, we change the name of the state $m_0^k$ to $ \left[\frac{(k+\beta')}{(k+\beta') V_{m_0}(\chi)+1}\right]$ if $V_{m_0^k}^{M'}(\chi) > 1-r^{M'}(n,m_0^k)$, and otherwise to $ \left[\frac{(k+\beta')}{(k+\beta')(1-r^M(n,m_0))+1}\right]$. Since $\vert S^{M'}\vert = \vert S^{M}\vert = \aleph_0$, it is possible to rename other states using natural number such that $S^{M''}=\mathbb{N}$. Thus, $M''$ is our desired model where for infinitely many states $m\in \mathbb{N}$, we have $\max\{1-r^{M''}(n,m), V_m^{M''}(\chi)\}\leq \frac{1}{m}$. Additionally, from (\ref{eq019}), we have $V_{m_0}^M(\chi) = V_m^{M''}(\chi)$ and $V_{m_0}^M(\psi) = V_{m}^{M''}(\psi)$, which implies $V_m(\psi) > V_m(\chi)$. Therefore, all renamed states $m\in \mathbb{N}$ (from $m_0^j$) belong to $B_n^{M''}$, indicating that $B_n^{M''}$ contains infinitely many states.
		
		\hfill $\blacksquare_\textit{Claim}$ \\
		\textbf{case 5.} $\varphi = B \psi$.
		Consider a reflexive model $M$ and $s\in S^M$ such that $(M,s)\nvDash \varphi$. Then, there exists a state $s'\in S^M$ such that $(M,s')\nvDash \psi$.
		By the induction hypothesis on $\psi$, there exists a reflexive model $M'=(\mathbb{N}, r, \pi)$ and $n\in \mathbb{N}$ such that for infinitely many $m\in\mathbb{N}$, we have
		
		\begin{equation}\label{eq_024}
			\max\{1-r^{M'}(n,m), V_m^{M'}(\psi)\}\leq\frac{1}{m}.
		\end{equation}
		Since $M'$ is reflexive
		we have 
		\begin{equation}\label{eq_025}
			V_m^{M'}(B\psi) \leq V_m^{M'}(\psi) \leq \frac{1}{m}.
		\end{equation}
		
		Thus, by utilizing (\ref{eq_024}) and (\ref{eq_025}), we can deduce that for infinitely many $m\in \mathbb{N}$, the inequality $\max\{1-r^{M'}(n,m), V_m^{M'}(B\psi)\}\leq \frac{1}{m}$ holds. 
	\end{proof}
\end{document}